\documentclass[a4paper,12pt, twoside]{article}

\newtheorem{thm}{Theorem}[section]
\newtheorem{prop}{Proposition}[section]
\newtheorem{rmk}{Remark}[section]
\newtheorem{lem}{Lemma}[section]
\newtheorem{defi}{Definition}[section]
\newtheorem{algo}{Algorithm}

\usepackage{authblk}
\usepackage{amsmath, amssymb}
\usepackage{amstext}
\usepackage[english]{babel}
\usepackage{xcolor}
\usepackage{amsfonts}
\usepackage{stackengine}
\usepackage{graphicx}
\usepackage{mathtools}
\usepackage{newtxtext}
\usepackage[varvw]{newtxmath}
\usepackage{times}
\usepackage{urwchancal}
\usepackage{bbold}
\usepackage{hyperref}
\usepackage{cleveref}
\usepackage{centernot}
\usepackage{subcaption}
\usepackage[export]{adjustbox}
\usepackage{mathrsfs}
\usepackage[md]{titlesec}
\usepackage[style=numeric-comp, backend=bibtex]{biblatex}
\usepackage{fancyhdr}
\usepackage{enumitem}
\usepackage{indentfirst}
\usepackage{float}
\usepackage{helvet} 

\newcommand{\veps}{v^{\varepsilon}} 
\newcommand{\peps}{p^{\varepsilon}} 
\newcommand{\Veps}{V_{\varepsilon}}
\newcommand{\Peps}{\Pi_{\varepsilon}}
\newcommand{\Vkh}{\mathcal{V}_{\varepsilon, k, h}}
\newcommand{\Pkh}{\mathit{\Pi}_{\varepsilon, k, h}}
\newcommand{\Zkh}{\mathcal{Z}_{\varepsilon, k, h}}
\newcommand{\Ueps}{U_{\varepsilon}}
\newcommand{\pheps}{p_{\varepsilon}}

\DeclareRobustCommand{\rchi}{{\mathpalette\irchi\relax}}
\newcommand{\irchi}[2]{\raisebox{\depth}{$#1\chi$}}

\newenvironment{proof}{\paragraph{Proof:}}{\hfill$\square$ \\}

\numberwithin{equation}{section}

\allowdisplaybreaks

\newcommand\blfootnote[1]{%
	\begingroup
	\renewcommand\thefootnote{}\footnote{#1}%
	\addtocounter{footnote}{-1}%
	\endgroup
}

\providecommand{\keywords}[1]
{
	\small	
	\textbf{\textit{Keywords:}} #1
}

\fancypagestyle{mypagestyle}{
	\fancyhf{}
	\fancyhead[OC]{\textit{J. Doghman}}
	\fancyhead[EC]{\textit{Numerical analysis of the stochastic Navier-Stokes}}
	\fancyhead[LO, LE]{\thepage}
	\fancyfoot{}
	
}
\title{Numerical approximation of the stochastic Navier-Stokes equations through artificial compressibility \blfootnote{The author is supported by a public grant as part of the Investissement d'avenir project [ANR-11-LABX-0056-LMH, LabEx LMH], and is part of the SIMALIN project [ANR-19-CE40-0016] of the French National Research Agency.\newline Email address: \texttt{jad.doghman@centralesupelec.fr}}}
\author[*]{Jad Doghman}
\affil[ ]{CNRS, Fédération de Mathématiques de CentraleSupélec FR 3487, Univ. Paris-Saclay, CentraleSupélec, 91190 Gif-sur-Yvette, France}
\date{}
\begin{document}
	\maketitle
	\begin{abstract}
		A constructive numerical approximation of the two-dimensional unsteady stochastic Navier-Stokes equations of an incompressible fluid is proposed via a pseudo-compressibility technique involving a parameter $\varepsilon$. Space and time are discretized through a finite element approximation and an Euler method. The convergence analysis of the suggested numerical scheme is investigated throughout this paper. It is based on a local monotonicity property permitting the convergence toward the unique strong solution of the Navier-Stokes equations to occur within the originally introduced probability space. Justified optimal conditions are imposed on the parameter $\varepsilon$ to ensure convergence within the best rate.
	\end{abstract}
	
	
	\section{Introduction}
	The first thought that springs to mind when it comes to the numerical simulation of the Navier-Stokes equations (NSEs) is the complexity of the occurring situation, which can be represented by turbulent behaviors and physical processes by which energy becomes not only unavailable but irrecoverable in any form. The notorious NSEs are widely-known for their essential role in modeling phenomena that emerge from aeronautical science, thermo-hydraulics, ocean dynamics, and so on. They read in this paper's context:
	\begin{equation}\label{eq Navier-Stokes}
		\begin{cases}
			\partial_{t}v - \nu\Delta v + [v\cdot\nabla]v + \nabla p = f + g(v)\dot{W}, \\
			div (v) = 0,\\
			v(0, \cdot) = v_{0},
		\end{cases}
	\end{equation}
	with $v = v(\omega, t, x)$ being the fluid velocity, $p = p(\omega, t, x)$ is the pressure, $f = f(\omega, t, x)$ embodies an external force, $g$ represents the diffusion coefficient, and $\nu > 0$ is the fluid kinematic viscosity. The term $W$ is regarded as a Wiener process admitting a trace-class covariance operator, with the notation $\dot{W} = \partial_{t}W(\omega, t,x)$.
	
	The present paper deals with numerical approximations of the two-dimensional incompressible NSEs driven by multiplicative noise, equipped with homogeneous Dirichlet boundary conditions, within a bounded polygonal domain of $\mathbb{R}^{2}$. Since the construction of divergence-free subspaces is not an effortless task (see for instance \cite{Bonizzoni2021, evans2013isogeometric, Neilan2016Michael}), the attention will be turned toward a variant of the underlying equations involving a pseudo-compressibility method, avoiding divergence-free fields, and owning the unique, strong solution of the NSEs when passing to the limit, under a few assumptions. To be more accurate, the model which will undergo the discretization later on satisfies:
	\begin{equation}\label{eq Navier-Stokes modified}
		\begin{cases}
			\partial_{t}\veps - \nu \Delta \veps + [\veps\cdot\nabla]\veps + \frac{1}{2}[div \ \veps]\veps + \nabla \peps = f + g(\veps)\dot{W},\\
			\varepsilon\partial_{t}\peps + div \ \veps = 0,\\
			\big(\veps(0, \cdot), \peps(0, \cdot)\big) = \left(v_{0}, p_{0}\right),
		\end{cases}
	\end{equation}
	where $\veps$ and $\peps$ are the associated fluid velocity and pressure, respectively. The parameter $\varepsilon > 0$ represents a small scale that will eventually tend to zero with the other discretization parameters to recover a solution to equations~\eqref{eq Navier-Stokes}, and $(v_{0}, p_{0})$ is the initial condition. The supplementary term $\frac{1}{2}[div \ \veps]\veps$ ensures the well-posedness of the model~\eqref{eq Navier-Stokes modified}, which is why it cannot be taken out. Notice that alternative configurations (also known as penalty methods) might have been possible, especially for the mass conservation equation of problem~\eqref{eq Navier-Stokes modified}. For instance,
	\begin{align*}
		&\varepsilon \peps + div \ \veps = 0, \\&
		\varepsilon\Delta\peps - div \ \veps = 0 \ \mbox{ with } \ \frac{\partial\peps}{\partial n} = 0, \\&
		\varepsilon\Delta\partial_{t}\peps - div \ \veps = 0 \ \mbox{ with } \ \frac{\partial}{\partial n}(\partial_{t}\peps) = 0, \mbox{and } \peps(0, \cdot) = p_{0}.
	\end{align*}
	The reader may refer to \cite{Temam1968}, \cite{Jie1996}, \cite{Temam1969}, and \cite{shen1992pressure} for thorough deterministic studies of the above mentioned techniques, including the one considered here.
	
	The mass conservation equation in problem~\eqref{eq Navier-Stokes modified} returns, in terms of regularity, good a priori estimates for the pressure $\peps$ (see \cite[Proposition 3.1]{Menaldi2008Sritharan}), which may be taken advantage of during the convergence rate analysis. In point of fact, the pressure's lack of time-regularity in equations~\eqref{eq Navier-Stokes} (see for instance \cite[Theorem 4.1]{pressure2003}) has a negative effect on the convergence rate of those equations, which appears through the time-rate $O(\Delta t^{-1})$, as it was illustrated in \cite[Corollary 4.2]{carelli2012rates}.
	
	Problem~\eqref{eq Navier-Stokes modified} was theoretically investigated in \cite{Menaldi2008Sritharan} where the authors conducted the existence and uniqueness properties of the associated solution. The proof technique therein consists of the local monotonicity property of the sum of the Stokes operator and the nonlinear term. A discrete version of this method will be considered in the present paper in order to demonstrate the convergence of the proposed numerical scheme and to avoid the Skorokhod theorem as well.
	
	This paper is split into five sections and is organized as follows. Section~\ref{section notations and preliminaries} provides the adequate preliminaries and configurations, including the required assumptions, solutions' definitions to problems~\eqref{eq Navier-Stokes}, \eqref{eq Navier-Stokes modified}, and the numerical scheme. Section~\ref{section main result} is devoted to giving the main theorem of this paper. Solvability, stability, and convergence of the numerical approximation are given in Section~\ref{section discussion} along with a linear version of the proposed numerical scheme. This same section grants a small analysis scope concerned with the best choice of the scale $\varepsilon$ in terms of the discretization parameters. The present paper ends with Section~\ref{section numerical exp.} whose role is to supply the reader with pieces of evidence through numerical experiments.
	
	\section{Notations, materials and algorithm}\label{section notations and preliminaries}
	Let $T > 0$ be a finishing time. Given a bounded polygonal domain $D \subset \mathbb{R}^{2}$ (for simplicity's sake), denote by $\partial D$ its boundary, and by $\vec{n}\colon \partial D \to \mathbb{R}^{2}$ its corresponding unit outward normal vector field. Function spaces in the Navier-Stokes framework are commonly denoted by $\mathbb{H}$ and $\mathbb{V}$ and are defined by
	\begin{align*}
		&\mathcal{V} \coloneqq \bigg\{z \in [C_{c}^{\infty}(D)]^{2} \ \big| \ div(z) = 0 \mbox{ in } D\bigg\},
		\\&\mathbb{H} \coloneqq \bigg\{z \in \left(L^{2}(D)\right)^{2} \ \big| \ div \ z = 0 \mbox{ a.e. in } D, \ z.\vec{n} = 0 \mbox{ a.e. on } \partial D\bigg\},
		\\&\mathbb{V} \coloneqq \bigg\{z \in \left(H^{1}_{0}(D)\right)^{2} \ \big| \ div \ z = 0 \mbox{ a.e. in } D\bigg\},
	\end{align*}
	where $C_{c}^{\infty}(D)$ denotes the space of $C^{\infty}(D)$ functions with compact support.
	The vector spaces will be henceforth indicated by blackboard bold letters for clarity's sake (e.g. $\mathbb{H}^{1} = \left(H^{1}(D)\right)^{2}$). The inner product of the Lebesgue space $\mathbb{L}^{2}$ and the duality product between $\mathbb{H}^{1}_{0}$ and $\mathbb{H}^{-1}$ are denoted by $\left(\cdot, \cdot\right)$ and $\langle \cdot, \cdot \rangle$, respectively. The parameter $\varepsilon$ of equation~\eqref{eq Navier-Stokes modified} satisfies all this paper long the condition $\varepsilon \leq 1$, the Gelfand triple $\left(\mathbb{H}_{0}^{1}, \mathbb{L}^{2}, \mathbb{H}^{-1}\right)$ will solely be employed, and the trilinear form
	\begin{equation*}
		\hat{b}(u, v, w) \coloneqq \left([u\cdot\nabla]v, w\right) + \frac{1}{2}\left([div \ u]v, w\right)
	\end{equation*}
	will be linked with equation~\eqref{eq Navier-Stokes modified}. Two operators can be associated with $\hat{b}$; the trilinear form $\displaystyle b(u, v, w) \coloneqq \left([u\cdot\nabla]v, w\right)$ that arises from the NSEs and the bilinear operator $\hat{B}\colon \mathbb{H}^{1}_{0}\times \mathbb{H}^{1}_{0} \to \mathbb{H}^{-1}$ which reads: $\displaystyle \langle\hat{B}(u, v), w\rangle = \hat{b}(u, v, w)$, for all $u, v, w \in \mathbb{H}_ {0}^{1}$. The upcoming proposition lists a few properties of the trilinear form $\hat{b}$ (cf. \cite{Jie1996}).
	\begin{prop}\label{prop trilinear form}
		\begin{enumerate}[label=(\roman*)]
			\item $\displaystyle \hat{b} \colon \mathbb{H}^{1}_{0}\times \mathbb{H}_{0}^{1}\times \mathbb{H}^{1}_{0} \to \mathbb{R}$ is continuous. \label{trilinear i}
			\item $\displaystyle \hat{b}(u, v, v) = 0$ for all $u, v \in \mathbb{H}^{1}_{0}$.\label{trilinear ii}
			\item $\displaystyle \left|\hat{b}(u, v, w)\right| \leq C_{D}\left|\left|u\right|\right|_{\mathbb{L}^{2}}^{\frac{1}{2}}\left|\left|\nabla u\right|\right|_{\mathbb{L}^{2}}^{\frac{1}{2}}\left|\left|\nabla v\right|\right|_{\mathbb{L}^{2}}\left|\left|\nabla w\right|\right|_{\mathbb{L}^{2}}$, for all $u, v, w \in \mathbb{H}_{0}^{1}$. \label{trilinear iii}
		\end{enumerate}
	\end{prop}
	Let $\left(\Omega, \mathcal{F}, \mathbb{P}\right)$ be a probability space endowed with a filtration $\{\mathcal{F}_{t}\}_{0 \leq t\leq T}$ such that $\mathcal{F}_{0}$ contains all the null sets and $\displaystyle \mathcal{F}_{t} = \bigcap_{s > t}\mathcal{F}_{s}$.
	Let $K$ be a separable Hilbert space equipped with a complete orthonormal basis $\{w_{k}\}_{k \geq 1}$, and $Q$ be a nuclear operator on $K$ such that $w_{k}$ is an eigenvector of $Q$ for all $k \geq 1$. The noise term $W$ will be considered hereafter as a $Q$-Wiener process and it is defined by $W(t, x) = \sum_{k \geq 1}\sqrt{q_{k}}\beta^{k}(t)w_{k}(x)$, where $\beta^{k}(\cdot), k \in \mathbb{N}$ is a sequence of independent and identically distributed real-valued Brownian motions, and $q_{k}$, $k \geq 1$ is the eigenvalue of $Q$ that is associated with $w_{k}$. With that said, the required assumptions are listed below.
	\paragraph{\textbf{Assumptions}}
	\begin{enumerate}[label=$(S_{\arabic*})$]
		\item $Q \colon K \to K$ is a symmetric positive definite nuclear operator. \label{S1}
		\item For $p \in [2, +\infty)$, $\displaystyle v_{0} \in L^{2^{p}}(\Omega; \mathbb{L}^{2})$ and $p_{0} \in L^{2^{p}}(\Omega; L^{2}(D))$ are $\mathcal{F}_{0}$-measurable.\label{S2}
		\item For $p \in [1, +\infty)$, $\displaystyle f \in L^{2^{p}}(\Omega; L^{2}(0,T; \mathbb{H}^{-1}))$ and $\displaystyle g \colon \mathbb{L}^{2} \to \mathscr{L}_{2}(K, \mathbb{L}^{2})$ satisfies
		\begin{align*}
			&\left|\left|g(u) - g(v)\right|\right|_{\mathscr{L}_{2}(\sqrt{Q}(K), \mathbb{L}^{2})} \leq L_{g}\left|\left|u - v\right|\right|_{\mathbb{L}^{2}}, \ \ \ \forall u, v \in \mathbb{L}^{2},\\&
			\left|\left|g(u)\right|\right|_{\mathscr{L}_{2}(K,\mathbb{L}^{2})} \leq K_{1} + K_{2}\left|\left|u\right|\right|_{\mathbb{L}^{2}}, \ \ \ \forall u \in \mathbb{L}^{2},
		\end{align*}
		for some positive time-independent constants $K_{1}, K_{2}, L_{g}$ such that $L_{g} \leq \sqrt{\frac{\nu}{2C^{2}_{P}}}$, where $C_{P}$ is the Poincar{\'e} constant. \label{S3}
	\end{enumerate}
	
	Throughout this paper, the writing $x\lesssim y$ designates $x \leq c y$ for a universal constant $c \geq 0$, the constant $C_{D}$ may vary from one calculation to another; however, it will depend only on the domain $D$, and finally the symbol $\mathscr{L}_{2}(X, Y)$ refers to the space of Hilbert-Schmidt operators from $X$ to $Y$, where $X$ and $Y$ are two Hilbert spaces.
	
	To derive high-order moment estimates, the following Jensen-type inequality will be needed hereafter: for all $J, p \in \mathbb{N}\backslash\{0\}$, and for all real-valued sequence $(\alpha_{n})_{n}$,
	\begin{equation}\label{eq sum-Jensen-type}
		\left(\sum_{n=1}^{J}\left|\alpha_{n}\right|\right)^{2^{p}} \leq 3^{2^{p}-1}\sum_{n=1}^{J}\left|\alpha_{n}\right|^{2^{p}}.
	\end{equation}
	The proof of inequality~\eqref{eq sum-Jensen-type} follows from an induction argument applied on $\left(\sum_{n=1}^{J}|\alpha_{n}|\right)^{2} \leq 3\sum_{n=1}^{J}|\alpha_{n}|^{2}$.
	\subsection{Concept of solutions}
	According to \cite{Menaldi2008Sritharan}, a solution to equations~\eqref{eq Navier-Stokes modified} satisfies the following definition.
	\begin{defi}\label{definition NSE slightly comp. sol}
		Given a filtered probability space $\left(\Omega, \mathcal{F}, \mathcal{F}_{t}, \mathbb{P}\right)$, a stochastic process $\veps$ is said to be a strong solution to equations~\eqref{eq Navier-Stokes modified} under assumptions \ref{S1}-\ref{S3} if it belongs to $L^{2}\left(\Omega; C([0,T]; \mathbb{L}^{2}) \cap L^{2}(0,T; \mathbb{H}^{1}_{0})\right)$, and it satisfies for all $t \in [0,T]$, $\mathbb{P}$-a.s.
		\begin{equation*}
			\begin{cases}
				\begin{aligned}
					&\left(\veps(t), \varphi\right) + \nu\int_{0}^{t}\left(\nabla \veps(s), \nabla\varphi\right)ds + \int_{0}^{t}\hat{b}(\veps(s), \veps(s), \varphi)ds - \int_{0}^{t}\left(\peps(s), div \ \varphi\right)ds \\&\hspace{20pt}= \left(v_{0}, \varphi\right) + \int_{0}^{t}\langle f(s), \varphi\rangle ds + \left(\int_{0}^{t}g(\veps(s))dW(s), \varphi\right), \ \ \forall \varphi \in \mathbb{H}^{1}_{0},
				\end{aligned}\\
				\left(\varepsilon\partial_{t}\peps(t) + div \ \veps(t), q\right) = 0, \ \ \forall q \in L^{2}(D),
			\end{cases}
		\end{equation*}
		along with the energy inequality
		\begin{equation*}
			\begin{aligned}
				\mathbb{E}\left[\sup\limits_{0 \leq t \leq T}\left(\left|\left|\veps(t)\right|\right|_{\mathbb{L}^{2}}^{p} + \varepsilon\left|\left|\peps(t)\right|\right|_{\mathbb{L}^{2}}^{p}\right)e^{-\delta t} + p\nu\int_{0}^{T}\left|\left|\nabla \veps(t)\right|\right|_{\mathbb{L}^{2}}^{2}\left|\left|\veps(t)\right|\right|^{p-2}_{\mathbb{L}^{2}}e^{-\delta t}dt\right] \leq \mathcal{C},
			\end{aligned}
		\end{equation*}
		for all $p \in [2, +\infty)$, $\varepsilon>0$, $\delta>0$, and for some constant $\mathcal{C} > 0$ depending on $\delta, p, T, v_{0}, p_{0}, f, K_{1}, K_{2}$ and $\varepsilon$.
	\end{defi}
	On the other hand, a solution to problem~\eqref{eq Navier-Stokes} in $2$D can be defined as follows.
	\begin{defi}\label{definition NSE sol}
		Assume \ref{S1}-\ref{S3} and let $T > 0$. A stochastic process $v$ on a given filtered probability space $\left(\Omega, \mathcal{F}, (\mathcal{F}_{t})_{0\leq t \leq T}, \mathbb{P}\right)$ is a strong solution to equations~\eqref{eq Navier-Stokes} if it belongs to $L^{2}\left(\Omega; C([0,T]; \mathbb{H}) \cap L^{2}(0,T; \mathbb{V})\right)$, and it fulfills for all $0 \leq t \leq T$, $\mathbb{P}$-a.s.
		\begin{equation*}
			\begin{aligned}
				&\left(v(t), \varphi\right) + \nu\int_{0}^{t}\left(\nabla v(s), \nabla \varphi\right)ds + \int_{0}^{t}\left([v(s)\cdot\nabla]v(s), \varphi\right)ds \\&= \left(v_{0}, \varphi\right) + \int_{0}^{T}\langle f(s), \varphi \rangle ds + \left(\int_{0}^{t}g(v(s))dW(s), \varphi\right), \ \ \forall \varphi \in \mathbb{V}.
			\end{aligned}
		\end{equation*}
	\end{defi}
	
	\subsection{Discretization}\label{subsection discretiztion}
	The time interval $[0,T]$ will be decomposed into $M \in \mathbb{N}\backslash\{0\}$ subintervals with equidistant nodes $\{t_{\ell}\}_{\ell=0}^{M} \eqqcolon \mathscr{I}_{k}$ for simplicity's sake. The corresponding step is denoted $k \coloneqq \frac{T}{M}$.
	
	The spatial domain $D$, which is assumed to be convex, bounded and polygonal, will be covered by a quasi-uniform triangulation $\mathcal{T}_{h}$, with $h$ being the diameters' maximum of all triangles. Let $\mathbb{H}_{h}$ be a subspace of $\mathbb{H}_{0}^{1}$ consisting of $[C(\bar{D})]^{2}$-valued piecewise polynomials over $\mathcal{T}_{h}$, and fulfilling for all $m \geq 2$:
	\begin{equation}\label{eq FE error estimate1}
		\inf\limits_{v_{h} \in \mathbb{H}_{h}}\left\{\left|\left|v - v_{h}\right|\right|_{\mathbb{L}^{2}} + h\left|\left|\nabla (v - v_{h})\right|\right|_{\mathbb{L}^{2}}\right\} \leq Ch^{m}\left|\left|v\right|\right|_{\mathbb{H}^{m}}, \ \ \forall v \in \mathbb{H}_{0}^{1}\cap \mathbb{H}^{m}.
	\end{equation}
	The quasi-uniformity of $\mathcal{T}_{h}$ permits the inverse inequality (cf. \cite[Lemma 4.5.3]{brenner2007mathematical}):
	\begin{equation}\label{eq inverse inequality}
		\left|\left|v_{h}\right|\right|_{\mathbb{H}^{\ell}} \leq \mathscr{C}h^{m-\ell}\left|\left|v_{h}\right|\right|_{\mathbb{H}^{m}}, \ \forall v_{h} \in \mathbb{H}_{h}, \ \forall \ 0 \leq m \leq \ell,
	\end{equation}
	for some $\mathscr{C} > 0$ independent of $h$.
	Let $L_{h}$ be a subspace of $L_{0}^{2}(D)$ consisting of $C(\bar{D})$ piecewise polynomial functions over $\mathcal{T}_{h}$, and satisfying for all $m \geq 1$:
	\begin{equation}\label{eq FE error estimate2}
		\inf\limits_{p_{h} \in L_{h}}\left|\left|p - p_{h}\right|\right|_{\mathbb{L}^{2}} \leq Ch^{m}\left|\left|p\right|\right|_{\mathbb{H}^{m}}, \ \ \forall p \in L_{0}^{2}(D) \cap H^{m}(D).
	\end{equation}
	For $(v, p) \in \mathbb{L}^{2}\times L^{2}(D)$, the associated orthogonal projections are denoted $\Pi_{h} \colon \mathbb{L}^{2} \to \mathbb{H}_{h}$ and $\rho_{h} \colon L^{2}(D) \to L_{h}$ and are defined by the following identities, respectively:
	\begin{equation}\label{def projection}
		\left(v - \Pi_{h}v, \varphi_{h}\right) = 0, \ \forall \varphi_{h} \in \mathbb{H}_{h} \ \mbox{ and }\ \left(p - \rho_{h}p, q_{h}\right) = 0, \ \forall q_{h} \in L_{h}.
	\end{equation}
	Thanks to the pseudo-compressibility method which is provided by equations~\eqref{eq Navier-Stokes modified}, the finite element pair $\left(\mathbb{H}_{h}, L_{h}\right)$ is not forced to satisfy the discrete LBB condition.
	
	For the sake of clarity, the notations $\varphi^{+}$ and $\varphi^{-}$ will designate throughout this paper piecewise constant functions with respect to time. For instance, $\varphi^{+}(t) \coloneqq \varphi^{m}, \forall t \in (t_{m-1}, t_{m}]$ and $\varphi^{-}(t) \coloneqq \varphi^{m-1}, \forall t \in [t_{m-1}, t_{m})$ for the a given sequence $\{\varphi^{m}\}_{m}$. The discrete derivation with respect to time will also intervene later on. For this purpose, the below proposition (cf. \cite[Appendix B]{brzezniak2013finite}) lists a few associated properties.
	\begin{prop}\label{prop discrete derivation}
		Given a sequence $\{\varphi^{m}\}_{m}$, the discrete derivative is defined by $d_{t}\varphi^{m} = \frac{\varphi^{m} - \varphi^{m-1}}{k}$, for all $m \in \{1, \dotsc, M\}$, and it fulfills the following assertions:
		\begin{enumerate}[label=(\roman*)]
			\item $\displaystyle d_{t}(\varphi^{+}\psi^{+}) = \varphi^{+}d_{t}\psi^{+} + \psi^{-}d_{t}\varphi^{+}$,
			\item $\displaystyle\int_{0}^{T}\varphi^{+}d_{t}\psi^{+}dt = \varphi^{+}(T)\psi^{+}(T) - \varphi^{-}(0)\psi^{-}(0) - \int_{0}^{T}\left(d_{t}\varphi^{+}\right)\psi^{-}dt$,
			\item $\displaystyle d_{t}e^{\varphi^{+}} = e^{\varphi^{-}}d_{t}\varphi^{+} + e^{\delta}\frac{\left(\varphi^{+} - \varphi^{-}\right)^{2}}{2k}$, for some $\delta \in (\varphi^{-}, \varphi^{+})$.
		\end{enumerate}
	\end{prop}
	
	Relying on Definition~\ref{definition NSE slightly comp. sol} and the space-time discretization, the numerical scheme which will be studied throughout the rest of this paper is given by:
	\begin{algo}\label{Algo}
		Let $m \in \{1, \dotsc, M\}$ and $(v^{0}_{h}, p^{0}_{h}) \in \mathbb{H}_{h} \times L_{h}$ be a starting point. For a given $\left(\Veps^{m-1}, \Peps^{m-1}\right) \in \mathbb{H}_{h}\times L_{h}$ such that $\left(\Veps^{0}, \Peps^{0}\right) \coloneqq (v_{h}^{0}, p^{0}_{h}) $, find $\left(\Veps^{m}, \Peps^{m}\right) \in \mathbb{H}_{h} \times L_{h}$ that satisfies
		\begin{equation*}
			\begin{cases}
				\begin{aligned}
					&\left(\Veps^{m} - \Veps^{m-1}, \varphi_{h}\right) + k\nu\left(\nabla \Veps^{m}, \nabla\varphi_{h}\right) + k\hat{b}(\Veps^{m}, \Veps^{m}, \varphi_{h}) - k\left(\Peps^{m}, div\varphi_{h}\right) \\&\hspace{20pt}= k\langle f^{m}, \varphi_{h} \rangle + \left(g(\Veps^{m-1})\Delta_{m}W, \varphi_{h}\right), \ \ \forall \varphi_{h} \in \mathbb{H}_{h},
				\end{aligned}\\
				\frac{\varepsilon}{k}\left(\Peps^{m} - \Peps^{m-1}, q_{h}\right) + \left(div\Veps^{m}, q_{h}\right) = 0, \ \ \forall q_{h} \in L_{h},
			\end{cases}
		\end{equation*}
		where for all $m \in \{1, \dotsc, M\}$, $f^{m} \coloneqq \frac{1}{k}\int_{t_{m-1}}^{t_{m}}f(t)dt$ and $\Delta_{m}W \coloneqq W(t_{m}) - W(t_{m-1})$.
	\end{algo}
	The initial datum $(v^{0}_{h}, p^{0}_{h})$ of Algorithm~\ref{Algo} is required to be uniformly bounded in $\mathbb{L}^{2}\times L^{2}(D)$ with respect to $h$. To this end, it suffices to consider $v_{h}^{0} = \Pi_{h}v_{0}$ and $p_{h}^{0} = \rho_{h}p_{0}$ because both projectors $\Pi_{h}$ and $\rho_{h}$ are stable in $L^{2}$ (cf. \cite{Douglas1974Lars}):
	\begin{equation}\label{eq projection stability}
		\left|\left|\Pi_{h}u\right|\right|_{\mathbb{L}^{2}} \leq \left|\left|u\right|\right|_{\mathbb{L}^{2}}, \ \forall u \in \mathbb{L}^{2} \ \mbox{ and } \ \left|\left|\rho_{h}q\right|\right|_{L^{2}} \leq \left|\left|q\right|\right|_{L^{2}}, \ \forall q \in L^{2}.
	\end{equation}
	Owing to \cite[Lemma III.4.5]{temam2001navier}, there holds
	\begin{equation}\label{eq f estimate}
		k\sum_{m=1}^{M}\left|\left|f^{m}\right|\right|^{2}_{\mathbb{H}^{-1}} \leq \int_{0}^{T}\left|\left|f(t)\right|\right|_{\mathbb{H}^{-1}}^{2}dt.
	\end{equation}
	\section{Main result}\label{section main result}
	\begin{thm}\label{theorem 1}
		For $T > 0$, let $\left(\Omega, \mathcal{F}, (\mathcal{F}_{t})_{0\leq t \leq T}, \mathbb{P}\right)$ be a filtered probability space, $D \subset \mathbb{R}^{2}$ be a polygonal domain, assumptions~\ref{S1}-\ref{S3} be satisfied, $\varepsilon > 0$ be a small scale, and $(k,h)$ be a finite positive pair such that $k$ corresponds to the step of an equidistant time partition $\mathscr{I}_{k}$ and $h$ is associated with a quasi-uniform triangulation $\mathcal{T}_{h}$ of the domain $D$. For a finite element pair $(\mathbb{H}_{h}, L_{h})$, let $(v_{h}^{0}, p_{h}^{0})$ belong to $\mathbb{H}_{h}\times L_{h}$ such that $\left|\left|(v_{h}^{0}, p_{h}^{0})\right|\right|_{\mathbb{L}^{2}\times L^{2}}$ is uniformly bounded in $h$. Then, Algorithm~\ref{Algo} has a solution $\{\Veps^{m}, \Peps^{m}\}_{m=1}^{M}$ satisfying Lemmata~\ref{lemma existence+measurability}, \ref{lemma iterates uniqueness}, \ref{lemma a priori estimates}. Further, if $v_{h}^{0} \to v_{0}$ in $L^{2}(\Omega; \mathbb{L}^{2})$ as $h \to 0$ then, as $\varepsilon, k, h \to 0$, Algorithm~\ref{Algo} converges toward the unique strong solution of equations~\eqref{eq Navier-Stokes} provided $\frac{k}{\varepsilon} \to 0$.
	\end{thm}	
	To obtain a solution to equations~\eqref{eq Navier-Stokes} from Algorithm~\ref{Algo}, the condition $\frac{k}{\varepsilon} \to 0$ in Theorem~\ref{theorem 1} can not be avoided on account of the finite element space $\mathbb{H}_{h}$ which does not accept divergence-free test functions (the reader may refer to section~\ref{subsection convergence} for more details). However, the aforementioned condition could be eliminated if the convergence of Algorithm~\ref{Algo} is carried out to achieve a solution to equations~\ref{eq Navier-Stokes modified}. In other words, $\varepsilon$ should be non-vanishing within this step (which is not the target of this paper). Afterwards, one can take advantage of \cite[Proposition 4.1]{Menaldi2008Sritharan} to gain a solution to equations~\eqref{eq Navier-Stokes}.
	\section{Discussion}\label{section discussion}
	\subsection{Existence and uniqueness of solutions}\label{section exist. and uniqu.}
	This section is devoted to giving existence and uniqueness properties to the discrete solution $\{(\Veps^{m}, \Peps^{m})\}_{m=1}^{M}$. The solvability of Algorithm~\ref{Algo} and the measurability of its iterates are handled first in the following lemma.
	\begin{lem}\label{lemma existence+measurability}
		Let $T > 0$ be fixed. Under assumptions~\ref{S1}-\ref{S3}, Algorithm~\ref{Algo} has at least one discrete solution. Moreover, for all $m \in \{1, \dotsc, M\}$, the processes $\Veps^{m} \colon \Omega \to \mathbb{H}_{h}$ and $\Peps^{m} \colon \Omega \to L_{h}$ are $\mathcal{F}_{t_{m}}$-measurable.
	\end{lem}
	\begin{proof}
		The solvability of Algorithm~\ref{Algo} can be proven by induction. Indeed, assume that iterates $\Veps^{\ell}$ and $\Peps^{\ell}$ exist for all $\ell \in \{1, \dotsc, m-1\}$. The existence of $(\Veps^{m}, \Peps^{m})$ is therefore the target. To this end, let $E \coloneqq \mathbb{H}^{1}_{0}\times L_{0}^{2}$, and $B_{\varepsilon} \colon E \to E$ be define by
		\begin{equation*}
			\begin{aligned}
				\left(B_{\varepsilon}(u, p), (v, q)\right)_{\mathbb{L}^{2}\times L^{2}} &= \left(u - \Veps^{m-1}(\omega), v\right) + k\nu\left(\nabla u, \nabla v\right) + k\hat{b}(u, u, v) - k\left(p, div \ v\right) - k\langle f^{m}, v\rangle \\&- \left(g(\Veps^{m-1}(\omega))\Delta_{m}W(\omega), v\right) + \varepsilon\left(p - \Peps^{m-1}(\omega), q\right) + k \left(div \ u, q\right),
			\end{aligned}
		\end{equation*}
		for all $(u,p), (v, q) \in E$, and for almost all $\omega \in \Omega$. The symbol $\left(\cdot, \cdot\right)_{\mathbb{L}^{2}\times L^{2}}$ denotes the $\mathbb{L}^{2}\times L^{2}(D)$-inner product. Thanks to Proposition~\ref{prop trilinear form}-\ref{trilinear i}, the continuity of $B_{\varepsilon}$ can be tackled easily. Through the application of Proposition~\ref{prop trilinear form}-\ref{trilinear ii}, the Poincar{\'e} and Young inequalities, estimate~\eqref{eq f estimate}, and assumption~\ref{S3}, one obtains
		\begin{equation*}
			\begin{aligned}
				&\left(B_{\varepsilon}(u, p), (u, p)\right)_{\mathbb{L}^{2}\times L^{2}} \geq \frac{1}{2}||u||_{\mathbb{L}^{2}}^{2} - \frac{1}{2}||\Veps^{m-1}||_{\mathbb{L}^{2}}^{2} + k\nu||\nabla u||_{\mathbb{L}^{2}}^{2} - k||f^{m}||_{\mathbb{H}^{-1}}||u||_{\mathbb{H}^{1}} \\&- ||g(\Veps^{m-1})||_{\mathscr{L}_{2}(K, \mathbb{L}^{2})}||\Delta_{m}W||_{K}||u||_{\mathbb{L}^{2}} + \frac{\varepsilon}{2}||p||_{L^{2}}^{2} - \frac{\varepsilon}{2}||\Peps^{m-1}||^{2}_{L^{2}} \geq \frac{1}{4}||u||^{2}_{\mathbb{L}^{2}} + \frac{\varepsilon}{2}||p||^{2}_{L^{2}} \\&- \frac{1}{2}||\Veps^{m-1}||^{2}_{\mathbb{L}^{2}} - \frac{\varepsilon}{2}||\Peps^{m-1}||^{2}_{L^{2}} - \frac{C_{D}^{2}}{4\nu}||f||^{2}_{L^{2}(0,T;\mathbb{H}^{-1})} - (K_{1} + K_{2}||\Veps^{m-1}||_{\mathbb{L}^{2}})^{2}||\Delta_{m}W||_{K}^{2} \geq 0,
			\end{aligned}
		\end{equation*}
		for all $(u, p) \in E_{h}(\omega) \coloneqq \left\{(v, q) \in \mathbb{H}_{h} \times L_{h} \ | \ ||v||_{\mathbb{L}^{2}} \geq \sqrt{S(\omega)}, \ ||q||_{L^{2}} \geq ||\Peps^{m-1}(\omega)||_{L^{2}}\right\}$, where $S(\omega) \coloneqq 2||\Veps^{m-1}(\omega)||^{2}_{\mathbb{L}^{2}} + \frac{C_{D}^{2}}{\nu}||f||^{2}_{L^{2}(0,T; \mathbb{H}^{-1})} + 4(K_{1} + K_{2}||\Veps^{m-1}(\omega)||_{\mathbb{L}^{2}})^{2}||\Delta_{m}W||_{K}^{2}$. Both $S(\omega)$ and $||\Peps^{m-1}(\omega)||_{L^{2}}$ are $\mathbb{P}$-a.s. finite, thanks to the induction supposition. With that said, the Brouwer fixed point theorem~\cite[Corollary IV.1.1]{girault2012finite} implies the existence of at least one $(u_{\omega}, p_{\omega}) \in \mathbb{H}_{h}\times L_{h}$ such that $B_{\varepsilon}(u_{\omega}, p_{\omega}) = (0, 0)$, $||u_{\omega}||_{\mathbb{L}^{2}} \leq \sqrt{S(\omega)}$ and $||p_{\omega}||_{L^{2}} \leq ||\Peps^{m-1}||_{L^{2}}$. Therewith, it suffices to set $(\Veps^{m}, \Peps^{m}) = (u_{\omega}, p_{\omega})$. On the other hand, the measurability of $\{(\Veps^{m}, \Peps^{m})\}_{m=1}^{M}$ can be also demonstrated by induction. The idea consists in expressing the newly obtained iterates $(u_{\omega}, p_{\omega})$ in terms of the existing ones. This can be done through a universally Borel-measurable selector function $\sigma \colon \mathbb{H}_{h}\times L_{h} \times K \to \mathbb{H}_{h}\times L_{h}$. For instance, $(u_{\omega}, p_{\omega}) = \sigma(\Veps^{m-1}, \Peps^{m-1}, \Delta_{m}W)$, and the $\mathcal{F}_{t_{m}}$-measurability arises from the Brownian increment $\Delta_{m}W$. The reader may refer to \cite[Page 744]{de2004semi} for a detailed approach.
	\end{proof}
	
	Lemma~\ref{lemma existence+measurability} dealt with the existence of a discrete solution which might not be unique. In point of fact, uniqueness in the whole probability set $\Omega$ does not seem to hold due to the nonlinearity interaction. Also, a contraction argument does not perform well in the discrete settings because the discrete time-derivative of an exponential function leads to a supplementary term which blocks the demonstration (see Proposition~\ref{prop discrete derivation}-(iii)). However, it can be proven that iterates' uniqueness holds true in a sample subset of $\Omega$ as demonstrated in the following lemma.
	\begin{lem}\label{lemma iterates uniqueness}
		Assume~\ref{S1}-\ref{S3} and let $\delta > 0$ be a small constant. Solutions $\{(\Veps^{m}, \Peps^{m})\}_{m=1}^{M}$ to Algorithm~\ref{Algo} are $\mathbb{P}$-almost surely unique within either of the following probability subsets:
		\begin{enumerate}[label=(\roman*)]
			\item $\displaystyle\Omega^{1}_{\delta} \coloneqq \left\{\omega \in \Omega \ | \ k\sum_{m=1}^{M}\left|\left|\Veps^{m}\right|\right|_{\mathbb{L}^{4}}^{4} \leq \frac{1}{\delta}\right\}$ provided that $\displaystyle\frac{1}{\nu^{3}\delta} \leq 4c_{0}^{3}$,
			\item $\displaystyle \Omega^{2}_{\delta} \coloneqq \left\{\omega \in \Omega \ | \ \max\limits_{1 \leq m \leq M}\left|\left|\Veps^{m}\right|\right|^{4}_{\mathbb{L}^{2}} \leq \frac{1}{\delta}\right\}$ provided that $\displaystyle\frac{k}{\nu^{3}\delta h^{2}} \leq \frac{2c_{0}^{3}}{\mathscr{C}^{2}}$,
		\end{enumerate}
		for some universal constant $c_{0} \in (0, 3^{-1}2^{\frac{2}{3}}]$. Furthermore, $\displaystyle\mathbb{P}(\Omega_{\delta}^{1}) \geq 1- \delta\mathbb{E}\left[k\sum_{m=1}^{M}\left|\left|\Veps^{m}\right|\right|^{4}_{\mathbb{L}^{4}}\right]$ and $\displaystyle\mathbb{P}(\Omega^{2}_{\delta}) \geq 1 - \delta\mathbb{E}\left[\max\limits_{1 \leq m \leq M}\left|\left|\Veps^{m}\right|\right|^{4}_{\mathbb{L}^{2}}\right]$.
	\end{lem}
	\begin{proof}
		Assume that $\{(\Veps^{m}, \Peps^{m})\}_{m=1}^{M}$ and $\{(U_{\varepsilon}^{m}, P_{\varepsilon}^{m})\}_{m=1}^{M}$ are solutions to Algorithm~\ref{Algo} starting from the same initial condition $(v_{h}^{0}, p_{h}^{0})$. For all $m \in \{0, 1, \dotsc, M\}$, let $Z_{\varepsilon}^{m} \coloneqq \Veps^{m} - U_{\varepsilon}^{m}$ and $Q_{\varepsilon}^{m} \coloneqq \Peps^{m} - P_{\varepsilon}^{m}$. Then, iterates $\{(Z_{\varepsilon}^{m}, Q_{\varepsilon}^{m})\}_{m=1}^{M}$ satisfy for all $m \in \{1, \dotsc, M\}$ and $\mathbb{P}$-a.s. the following equations
		\begin{equation}\label{calc1}
			\begin{cases}
				\begin{aligned}
					&\left(Z_{\varepsilon}^{m} - Z_{\varepsilon}^{m-1}, \varphi_{h}\right) + k\nu\left(\nabla Z_{\varepsilon}^{m}, \nabla \varphi_{h}\right) + k\langle \hat{B}(\Veps^{m}, \Veps^{m}) - \hat{B}(U_{\varepsilon}^{m}, U_{\varepsilon}^{m}), Z_{\varepsilon}^{m}\rangle - k\left(Q_{\varepsilon}^{m}, div\varphi_{h}\right) \\&\hspace{1cm}= \left([g(\Veps^{m-1}) - g(U_{\varepsilon}^{m-1})]\Delta_{m}W, \varphi_{h}\right), \ \ \forall \varphi_{h} \in \mathbb{H}_{h},
				\end{aligned}\\
				\frac{\varepsilon}{k}\left(Q_{\varepsilon}^{m} - Q_{\varepsilon}^{m-1}, q_{h}\right) + \left(div Z_{\varepsilon}^{m}, q_{h}\right) = 0, \ \ \forall q_{h} \in L_{h}.
			\end{cases}
		\end{equation}
		Observe that the stochastic term in equation~\eqref{calc1} can be eliminated if one had $\Veps^{m-1} = U^{m-1}_{\varepsilon}$. Since $U^{0}_{\varepsilon} = \Veps^{0} = v_{h}^{0}$, an induction argument seems to be legitimate. Indeed, for $m=1$ and $(\varphi_{h}, q_{h}) = (Z_{\varepsilon}^{1}, Q_{\varepsilon}^{1})$, equation~\eqref{calc1} turns into
		\begin{equation}\label{calc2}
			\begin{aligned}
				&||Z_{\varepsilon}^{1}||_{\mathbb{L}^{2}}^{2} + \varepsilon||Q^{1}_{\varepsilon}||_{L^{2}}^{2} + k\nu||\nabla Z_{\varepsilon}^{1}||_{\mathbb{L}^{2}}^{2} = k\langle\hat{B}(U_{\varepsilon}^{1}, U_{\varepsilon}^{1}) - \hat{B}(\Veps^{1}, \Veps^{1}), Z_{\varepsilon}^{1}\rangle \\&\leq 2k||\nabla Z_{\varepsilon}^{1}||^{\frac{3}{2}}_{\mathbb{L}^{2}}||Z_{\varepsilon}^{1}||_{\mathbb{L}^{2}}^{\frac{1}{2}}||\Veps^{1}||_{\mathbb{L}^{4}} \leq 3.2^{-\frac{2}{3}}c_{0}k\nu||\nabla Z_{\varepsilon}^{1}||^{2}_{\mathbb{L}^{2}} + \frac{k}{4\nu^{3}c_{0}^{3}}||Z_{\varepsilon}^{1}||_{\mathbb{L}^{2}}^{2}||\Veps^{1}||_{\mathbb{L}^{4}}^{4},
			\end{aligned}
		\end{equation}
		where the first inequality employs the estimate $\left|\langle \hat{B}(u, u) - \hat{B}(v, v), z \rangle\right| \leq 2||\nabla(u-v)||_{\mathbb{L}^{2}}^{3/2}||u-v||^{1/2}_{\mathbb{L}^{2}}||z||_{\mathbb{L}^{4}}$ for all $u, v, z \in \mathbb{H}_{0}^{1}$ (see the proof in \cite[Lemma 2.3]{Menaldi2008Sritharan}), and the second inequality uses Young's inequality for some constant $c_{0} \in (0, 3^{-1}2^{\frac{2}{3}}]$. Subsequently, equation~\eqref{calc2} becomes
		\begin{equation}\label{calc3}
			\left(1-k4^{-1}\nu^{-3}c_{0}^{-3}||\Veps^{1}||^{4}_{\mathbb{L}^{4}}\right)||Z_{\varepsilon}^{1}||_{\mathbb{L}^{2}}^{2} + \varepsilon||Q_{\varepsilon}^{1}||^{2}_{L^{2}} \leq 0.
		\end{equation}
		One way to obtain uniqueness is by multiplying equation~\eqref{calc3} by the indicator function $\mathbb{1}_{\Omega_{\delta}^{1}}$ which grants $(1 - 4^{-1}\nu^{-3}c_{0}^{-3}\delta^{-1})\mathbb{1}_{\Omega_{\delta}^{1}}||Z_{\varepsilon}^{1}||_{\mathbb{L}^{2}}^{2} + \varepsilon\mathbb{1}_{\Omega_{\delta}^{1}}||Q_{\varepsilon}^{1}||_{L^{2}}^{2} \leq 0$. It follows that $Z_{\varepsilon}^{1} = Q_{\varepsilon}^{1} = 0$ a.e. in $D$ and $\mathbb{P}$-a.s. in $\Omega_{\delta}^{1}$ provided that the coefficient of $||Z_{\varepsilon}^{1}||^{2}_{\mathbb{L}^{2}}$ is positive. The second way for uniqueness consists in multiplying equation~\eqref{calc3} by $\mathbb{1}_{\Omega_{\delta}^{2}}$ after employing the inverse estimate~\eqref{eq inverse inequality}. That is, $||\Veps^{1}||^{4}_{\mathbb{L}^{4}} \leq 2||\Veps^{1}||^{2}_{\mathbb{L}^{2}}||\nabla \Veps^{1}||_{\mathbb{L}^{2}}^{2} \leq 2\mathscr{C}^{2}h^{-2}||\Veps^{1}||^{4}_{\mathbb{L}^{2}}$, where the first inequality is due to Ladyzhenskaya (see \cite[Lemma I.1]{Ladyzhenskaya1964}). Therefore, equation~\eqref{calc3} turns into $(1 - 2^{-1}\nu^{-3}c_{0}^{-3}\mathscr{C}^{2}\delta^{-1}h^{-2}k)\mathbb{1}_{\Omega_{\delta}^{2}}||Z_{\varepsilon}^{1}||_{\mathbb{L}^{2}}^{2} + \varepsilon\mathbb{1}_{\Omega_{\delta}^{2}}||Q_{\varepsilon}^{1}||_{L^{2}}^{2} \leq 0$ which implies $Z_{\varepsilon}^{1} = Q_{\varepsilon}^{1} = 0$ a.e. in $D$ and $\mathbb{P}$-a.s. in $\Omega_{\delta}^{2}$ provided the coefficient of $||Z_{\varepsilon}^{1}||^{2}_{\mathbb{L}^{2}}$ is positive. With that being said, it suffices to assume that $Z^{m-1}_{\varepsilon} = Q_{\varepsilon}^{m - 1} = 0$ a.e. in $D$, $\mathbb{P}$-a.s. in either $\Omega_{\delta}^{1}$ or $\Omega_{\delta}^{2}$, and re-apply the same technique to obtain a similar result for the rank $m$. Finally, estimates of $\mathbb{P}(\Omega_{\delta}^{1})$ and $\mathbb{P}(\Omega_{\delta}^{2})$ derive from the Markov inequality.
	\end{proof}
	\begin{rmk}
		Picking between $\Omega_{\delta}^{1}$ and $\Omega_{\delta}^{2}$ in Lemma~\ref{lemma iterates uniqueness} depends on the choice of the viscosity $\nu$. Observe that the condition $\frac{1}{\nu^{3}\delta} \leq 4c_{0}^{3}$ does not allow $\delta$ to be small when $\nu$ is tiny. Therewith, choosing $\nu$ large (resp. small) corresponds to $\Omega_{\delta}^{1}$ (resp. $\Omega_{\delta}^{2}$). Moreover, lower bounds associated with $\mathbb{P}(\Omega_{\delta}^{1})$ and $\mathbb{P}(\Omega_{\delta}^{2})$ in Lemma~\ref{lemma iterates uniqueness} are finite as illustrated in Lemma~\ref{lemma a priori estimates}. It is worth mentioning that $\mathbb{E}\left[k\sum_{m=1}^{M}\left|\left|\Veps^{m}\right|\right|^{4}_{\mathbb{L}^{4}}\right] \lesssim \mathbb{E}\left[\max\limits_{1 \leq m \leq M}\left|\left|\Veps^{m}\right|\right|_{\mathbb{L}^{2}}^{4}\right]^{\frac{1}{2}}\mathbb{E}\left[\left(k\sum_{m=1}^{M}\left|\left|\nabla \Veps^{m}\right|\right|_{\mathbb{L}^{2}}^{2}\right)^{2}\right]^{\frac{1}{2}}$.
	\end{rmk}
	\subsection{A priori bounds and convergence}\label{section estimates and convergence}
	The first part of this section is dedicated to achieving stability of Algorithm~\ref{Algo}, whose convergence toward the unique solution of equations~\eqref{eq Navier-Stokes} is handled in the second part.
	\subsubsection{A priori bounds}
	\begin{lem}\label{lemma a priori estimates}
		Let $p \in [2, +\infty) \cap \mathbb{N}$ be fixed and assumptions \ref{S1}-\ref{S3} be satisfied. Then, iterates $\{(\Veps^{m}, \Peps^{m})\}_{m=1}^{M}$ of Algorithm~\ref{Algo} fulfill the following estimates:
		\begin{enumerate}[label=(\roman*)]
			\item $\displaystyle\mathbb{E}\bigg[\max\limits_{1 \leq m \leq M}\left|\left|\Veps^{m}\right|\right|^{2}_{\mathbb{L}^{2}} + k\nu\sum_{m=1}^{M}\left|\left|\nabla \Veps^{m}\right|\right|^{2}_{\mathbb{L}^{2}} + \sum_{m=1}^{M}\left|\left|\Veps^{m} - \Veps^{m-1}\right|\right|^{2}_{\mathbb{L}^{2}}\bigg] \leq C_{1}$,
			\item $\displaystyle \mathbb{E}\bigg[\max\limits_{1 \leq m \leq M}\left|\left|\Peps^{m}\right|\right|^{2}_{L^{2}} + \sum_{m=1}^{M}\left|\left|\Peps^{m} - \Peps^{m-1}\right|\right|^{2}_{L^{2}}\bigg] \leq \frac{C_{1}}{\varepsilon},$
			\item $\displaystyle \mathbb{E}\bigg[\max\limits_{1 \leq m \leq M}\left|\left|\Veps^{m}\right|\right|_{\mathbb{L}^{2}}^{2^{p}} + \Big(k\nu\sum_{m=1}^{M}\left|\left|\nabla\Veps^{m}\right|\right|^{2}_{\mathbb{L}^{2}}\Big)^{2^{p-1}} + \Big(\sum_{m=1}^{M}\left|\left|\Veps^{m} - \Veps^{m-1}\right|\right|^{2}_{\mathbb{L}^{2}}\Big)^{2^{p-1}}\bigg] \leq C_{p}$,
			\item $\displaystyle \mathbb{E}\bigg[\max\limits_{1 \leq m \leq M}\left|\left|\Peps^{m}\right|\right|_{L^{2}}^{2^{p}} + \Big(\sum_{m=1}^{M}\left|\left|\Peps^{m} - \Peps^{m-1}\right|\right|_{L^{2}}^{2}\Big)^{2^{p-1}}\bigg] \leq \varepsilon^{-2^{p-1}}C_{p}$,
		\end{enumerate}
		for some constant $C_{p} \geq 0$ depending only on $||v_{0}||_{L^{2^{p}}(\Omega; \mathbb{L}^{2})}, ||p_{0}||_{L^{2^{p}}(\Omega; L^{2})}, D, \nu, ||f||_{L^{2^{p}}(\Omega; L^{2}(0,T; \mathbb{H}^{-1}))}$, $T, Tr(Q), K_{1}, p$ and $K_{2}$, with $C_{1} = C_{p=1}$.
	\end{lem}
	\begin{proof}
		Replace $(\varphi_{h}, q_{h})$ by $(\Veps^{m}, \Peps^{m})$ in Algorithm~\ref{Algo} and employ the identity $(a-b, a) = \frac{1}{2}(||a||^{2}_{\mathbb{L}^{2}} - ||b||^{2}_{\mathbb{L}^{2}} - ||a - b||^{2}_{\mathbb{L}^{2}})$ together with Proposition~\ref{prop trilinear form}-\ref{trilinear ii}, Cauchy-Schwarz, Poincar{\'e} and Young's inequalities:
		\begin{equation}\label{calc4}
			\begin{aligned}
				&\frac{1}{2}||\Veps^{m}||_{\mathbb{L}^{2}}^{2} - \frac{1}{2}||\Veps^{m-1}||^{2}_{\mathbb{L}^{2}} + \frac{1}{4}||\Veps^{m} - \Veps^{m-1}||_{\mathbb{L}^{2}}^{2} + \frac{\varepsilon}{2}\left(||\Peps^{m}||_{L^{2}}^{2} - ||\Peps^{m-1}||^{2}_{L^{2}} + ||\Peps^{m} - \Peps^{m-1}||^{2}_{L^{2}}\right) \\&+ \frac{k\nu}{2}||\nabla \Veps^{m}||^{2}_{\mathbb{L}^{2}} \leq \frac{C_{D}^{2}k}{2\nu}||f^{m}||^{2}_{\mathbb{H}^{-1}} + ||g(\Veps^{m-1})\Delta_{m}W||_{\mathbb{L}^{2}}^{2} + \left(g(\Veps^{m-1})\Delta_{m}W, \Veps^{m-1}\right).
			\end{aligned}
		\end{equation}
		Summing equations~\eqref{calc4} over $m$ from $1$ to $\ell \in \{1, \dotsc, M\}$, then applying the mathematical expectation, condition $\varepsilon \leq 1$, estimates~\eqref{eq projection stability} and \eqref{eq f estimate} yield
		\begin{equation}\label{calc5}
			\begin{aligned}
				&\mathbb{E}\Big[||\Veps^{\ell}||^{2}_{\mathbb{L}^{2}} + \varepsilon||\Peps^{\ell}||^{2}_{L^{2}} + k\nu\sum_{m=1}^{\ell}||\nabla \Veps^{m}||^{2}_{\mathbb{L}^{2}} + \sum_{m=1}^{\ell}\Big(\frac{1}{2}||\Veps^{m} - \Veps^{m-1}||^{2}_{\mathbb{L}^{2}} + \varepsilon||\Peps^{m} - \Peps^{m-1}||_{L^{2}}^{2}\Big)\Big] \\&\leq \mathbb{E}\Big[||v_{0}||^{2}_{\mathbb{L}^{2}} + ||p_{0}||^{2}_{L^{2}} + \frac{C_{D}^{2}}{\nu}||f||^{2}_{L^{2}(0,T; \mathbb{H}^{-1})} + 2\sum_{m=1}^{\ell}||g(\Veps^{m-1})\Delta_{m}W||^{2}_{\mathbb{L}^{2}}\Big],
			\end{aligned}
		\end{equation}
		where the mathematical expectation of last term in equation~\eqref{calc4} vanishes due to the $\mathcal{F}_{t_{m-1}}$-measurability of $\Veps^{m-1}$ together with assumption~\ref{S3}. On the other hand, the last term of inequality~\eqref{calc5} can be handled through the It{\^o} isometry and assumption~\ref{S3} as follows:
		\begin{equation}\label{calc5'}
			\begin{aligned}
				&\mathbb{E}\left[||g(\Veps^{m-1})\Delta_{m}W||^{2}_{\mathbb{L}^{2}}\right] = \mathbb{E}\left[\Big|\Big|\int_{t_{m-1}}^{t_{m}}g(\Veps^{m-1})dW(t)\Big|\Big|^{2}_{\mathbb{L}^{2}}\right] = k\mathbb{E}\left[\Big|\Big|g(\Veps^{m-1})\Big|\Big|^{2}_{\mathscr{L}_{2}(\sqrt{Q}(K), \mathbb{L}^{2})}\right] \\&\leq 2kTr(Q)K_{1}^{2} + 2kTr(Q)K_{2}^{2}\mathbb{E}\left[||\Veps^{m-1}||^{2}_{\mathbb{L}^{2}}\right].
			\end{aligned}
		\end{equation}
		Thus, the discrete Gr{\"o}nwall inequality implies
		\begin{equation}\label{calc6}
			\begin{aligned}
				&\max\limits_{1 \leq m \leq M}\mathbb{E}\left[||\Veps^{m}||_{\mathbb{L}^{2}}^{2} + \varepsilon||\Peps^{m}||_{L^{2}}^{2}\right] + \sum_{m=1}^{M}\mathbb{E}\left[k\nu||\nabla \Veps^{m}||_{\mathbb{L}^{2}}^{2} + \frac{1}{2}||\Veps^{m} - \Veps^{m-1}||_{\mathbb{L}^{2}}^{2}\right] \\& + \mathbb{E}\left[\sum_{m=1}^{M}\varepsilon||\Peps^{m} - \Peps^{m-1}||_{L^{2}}^{2}\right]\leq C_{1},
			\end{aligned}
		\end{equation}
		where $C_{1}>0$ depends only on $||v_{0}||_{L^{2}(\Omega; \mathbb{L}^{2})}, ||p_{0}||_{L^{2}(\Omega; L^{2})}, D, \nu, ||f||_{L^{2}(\Omega; L^{2}(0,T; \mathbb{H}^{-1}))}, T, Tr(Q), K_{1}$ and $K_{2}$. To terminate the proof of estimates $(i)$ and $(ii)$, it suffices to reconsider equation~\eqref{calc4}, sum it over $m$ from $1$ to $\ell \in \{1, \dotsc, M\}$, take the maximum over $\ell$, then apply the mathematical expectation to get
		\begin{equation}\label{calc7}
			\begin{aligned}
				&\mathbb{E}\Big[\max\limits_{1 \leq \ell \leq M}(||\Veps^{\ell}||^{2}_{\mathbb{L}^{2}} + \varepsilon||\Peps^{\ell}||^{2}_{L^{2}})\Big] \leq \mathbb{E}\Big[||v_{0}||_{\mathbb{L}^{2}}^{2} + ||p_{0}||^{2}_{L^{2}} + C_{D}^{2}\nu^{-1}||f||^{2}_{L^{2}(0,T; \mathbb{H}^{-1})} \\&+ 2\sum_{m=1}^{M}||g(\Veps^{m-1})\Delta_{m}W||^{2}_{\mathbb{L}^{2}} + 2\max\limits_{1 \leq \ell \leq M}\sum_{m=1}^{\ell}\left(g(\Veps^{m-1})\Delta_{m}W, \Veps^{m-1}\right)\Big].
			\end{aligned}
		\end{equation}
		The penultimate term is estimated in inequality~\eqref{calc5'}. The last term is controlled by
		\begin{equation*}
			\begin{aligned}
				&\lesssim 6\mathbb{E}\left[\left(k\sum_{m=1}^{M}||g(\Veps^{m-1})||^{2}_{\mathscr{L}_{2}(K, \mathbb{L}^{2})}||\Veps^{m-1}||^{2}_{\mathbb{L}^{2}}\right)^{\frac{1}{2}}\right] \\&\leq \frac{3}{4}\mathbb{E}\left[\max\limits_{1 \leq \ell \leq M}||\Veps^{\ell}||^{2}_{\mathbb{L}^{2}}\right] + \mathbb{E}\left[\frac{3}{4}||v_{h}^{0}||^{2}_{\mathbb{L}^{2}} + 3k\sum_{m=1}^{M}(K_{1}^{2} + K_{2}^{2}||\Veps^{m-1}||^{2}_{\mathbb{L}^{2}})\right],
			\end{aligned}
		\end{equation*}
		where Young's inequality and assumption~\ref{S3} are used together with the Davis inequality which is applicable since the integrand can be considered as a simple function with respect to time. Obviously, the first term on the right-hand side must be absorbed into the left side of equation~\eqref{calc7} and the remaining terms can be readily controlled through estimates~\eqref{eq projection stability} and \eqref{calc6}. This completes the proof of assertions $(i)$ and $(ii)$. Estimates $(iii)$ and $(iv)$ can be demonstrated as follows: let $p \geq 2$ be an integer. Summing equation~\eqref{calc4} over $m$ from $1$ to $\ell \in \{1, \dotsc, M\}$, making use of estimates~\eqref{eq projection stability},\eqref{eq f estimate}, raising both sides to the power $2^{p-1}$, then employing inequality~\eqref{eq sum-Jensen-type} multiple times yield
		\begin{equation}\label{calc8}
			\begin{aligned}
				&\max\limits_{1 \leq \ell \leq M}(||\Veps^{\ell}||^{2^{p}}_{\mathbb{L}^{2}} + \varepsilon^{2^{p-1}}||\Peps^{\ell}||^{2^{p}}_{L^{2}}) + \Big(\sum_{m=1}^{M}(||\Veps^{m} - \Veps^{m-1}||_{\mathbb{L}^{2}}^{2} + \varepsilon||\Peps^{m} - \Peps^{m-1}||^{2}_{L^{2}})\Big)^{2^{p-1}} \\&+ \Big(k\nu\sum_{m=1}^{M}||\nabla \Veps^{m}||_{\mathbb{L}^{2}}^{2}\Big)^{2^{p-1}} \lesssim ||v_{0}||^{2^{p}}_{\mathbb{L}^{2}} + ||p_{0}||^{2^{p}}_{L^{2}} + C_{D}\nu^{-2^{p-1}}||f||_{L^{2}(0,T; \mathbb{H}^{-1})}^{2^{p}} \\&+ \sum_{m=1}^{M}||g(\Veps^{m-1})\Delta_{m}W||_{\mathbb{L}^{2}}^{2^{p}} + \Big(\max\limits_{1 \leq \ell \leq M}\sum_{m=1}^{\ell}\left(g(\Veps^{m-1})\Delta_{m}W, \Veps^{m-1}\right)\Big)^{2^{p-1}}.
			\end{aligned}
		\end{equation}
		The mathematical expectation of the penultimate term is estimated through the Burkholder-Davis-Gundy inequality as follows:
		\begin{equation}\label{calc9}
			\begin{aligned}
				&\sum_{m=1}^{M}\mathbb{E}\left[\Big|\Big|\int_{t_{m-1}}^{t_{m}}g(\Veps^{m-1})dW(t)\Big|\Big|_{\mathbb{L}^{2}}^{2^{p}}\right] \lesssim \sum_{m=1}^{M}\left(\int_{t_{m-1}}^{t_{m}}\mathbb{E}\left[||g(\Veps^{m-1})||_{\mathscr{L}_{2}(K, \mathbb{L}^{2})}^{2^{p}}\right]^{\frac{1}{2^{p-1}}}dt\right)^{2^{p-1}} \\&= k^{2^{p-1}}\sum_{m=1}^{M}\mathbb{E}\left[||g(\Veps^{m-1})||^{2^{p}}_{\mathscr{L}_{2}(K, \mathbb{L}^{2})}\right] \lesssim K_{1}^{2^{p}}T + k^{2^{p-1}}\sum_{m=1}^{M}K_{2}^{2^{p}}\mathbb{E}\left[||\Veps^{m-1}||^{2^{p}}_{\mathbb{L}^{2}}\right],
			\end{aligned}
		\end{equation}
		thanks to inequality~\eqref{eq sum-Jensen-type} and assumption~\ref{S3}. The last term of equation~\eqref{calc8} can be controlled by
		\begin{equation*}
			\begin{aligned}
				&\lesssim \mathbb{E}\left[\Big(k\sum_{m=1}^{M}||g(\Veps^{m-1})||^{2}_{\mathscr{L}_{2}(K, \mathbb{L}^{2})}||\Veps^{m-1}||_{\mathbb{L}^{2}}^{2}\Big)^{2^{p-2}}\right] \leq \frac{1}{4}\mathbb{E}\Big[\max\limits_{1 \leq \ell \leq M}||\Veps^{\ell}||^{2^{p}}_{\mathbb{L}^{2}}\Big] \\&+ \mathbb{E}\Big[\frac{1}{4}||v_{h}^{0}||_{\mathbb{L}^{2}}^{2^{p}} + 3^{2^{p-1}-1}k^{2^{p-1}}\sum_{m=1}^{M}||g(\Veps^{m-1})||^{2^{p}}_{\mathscr{L}_{2}(K, \mathbb{L}^{2})}\Big]
			\end{aligned}
		\end{equation*}
		where the Burkholder-Davis-Gundy inequality is employed (see \cite[Theorem 2.4]{brzezniak1997stochastic}) together with estimate~\eqref{eq sum-Jensen-type} and Young's inequality. By virtue of assumption~\ref{S3}, the last term can be bounded by almost the same right-hand side of equation~\eqref{calc9}. Putting it all together and applying the discrete Gr{\"o}nwall inequality to equation~\eqref{calc8} complete the proof.
	\end{proof}
	\subsubsection{Convergence}\label{subsection convergence}
	Stability properties that were derived in Lemma~\ref{lemma a priori estimates} will play a crucial role in this part, especially to offer convergence results to $\{(\Veps^{m}, \Peps^{m})\}_{m=1}^{M}$ as $\varepsilon, k,h \to 0$ . For this purpose, a few new notations must be summoned along with one important lemma consisting of a monotonicity property that allows the convergence of Algorithm~\ref{Algo} toward equations~\eqref{eq Navier-Stokes} to occur. For all $m \in \{1, \dotsc, M\}$, the new notations read:
	\begin{align*}
		&\left(\Vkh^{+}(t), \Pkh^{+}(t)\right) \coloneqq \left(\Veps^{m}, \Peps^{m}\right), \ \forall t \in (t_{m-1}, t_{m}], 
		\\& \left(\Vkh^{-}(t), \Pkh^{-}(t)\right) \coloneqq \left(\Veps^{m-1}, \Peps^{m-1}\right), \ \forall t \in [t_{m-1}, t_{m}).
	\end{align*}
	There will also be similar notations in the upcoming part such as $f^{+}$ and $r^{-}$; the reader may refer to section~\ref{subsection discretiztion} for an adequate definition.
	Note that it is not mandatory for $\varepsilon$ to be dependent on the discretization parameters $k$ and $h$. If so, it suffices that $\varepsilon = \varepsilon(k,h) \to 0$ as $k,h \to 0$. It is worth mentioning the assumption $\frac{k}{\varepsilon} \to 0$ as $k, h, \varepsilon \to 0$, which will be imposed later on to achieve the convergence toward the solution of equations~\eqref{eq Navier-Stokes}. Such a hypothesis arises from the chosen finite element space $\mathbb{H}_{h}$, which does not offer divergence-free test functions for the sake of eliminating the term $\langle \nabla\Peps^{m}, \varphi_{h}\rangle$, which appears during the calculations hereafter. Another way to justify this hypothesis is through the following proposition.
	\begin{prop}
		Let $\left\{\left(\Veps^{m}, \Peps^{m}\right)\right\}_{m=1}^{M}$ be the iterates of Algorithm~\ref{Algo}. Then,
		\begin{enumerate}[label=(\roman*)]
			\item $\displaystyle \mathbb{E}\left[\max\limits_{1 \leq m \leq M}\varepsilon\left|\left|\Peps^{m}\right|\right|_{L^{2}}^{2}\right] \leq \varepsilon\left|\left|p_{0}\right|\right|^{2}_{L^{2}(\Omega; L^{2})} + c\frac{k}{\varepsilon}$,
			\item $\displaystyle \mathbb{E}\left[\sum_{m=1}^{M}\varepsilon\left|\left|\Peps^{m} - \Peps^{m-1}\right|\right|^{2}_{L^{2}}\right] \leq c\frac{k}{\varepsilon}$,
		\end{enumerate}
		for some genuine constant $c > 0$ independent of $k, h$ and $\varepsilon$.
	\end{prop}
	\begin{proof}
		Let $q \in L^{2}(D)\backslash \{0\}$. By identity~\eqref{def projection}, it holds that $(\Peps^{m} - \Peps^{m-1}, q) = \left(\Peps^{m} - \Peps^{m-1}, \rho_{h}q\right)$. Therefore, using Algorithm~\ref{Algo}, one obtains
		\begin{equation}\label{calc17}
			\begin{aligned}
				\varepsilon(\Peps^{m} - \Peps^{m-1}, q) = -k(div\Veps^{m}, \rho_{h}q) \lesssim k||\nabla \Veps^{m}||_{\mathbb{L}^{2}}||q||_{L^{2}},
			\end{aligned}
		\end{equation}
		thanks to the Cauchy-Schwarz inequality and the stability of $\rho_{h}$ in $L^{2}(D)$. Summing both sides over $m$ from $1$ to an arbitrary $\ell \in \{1, \dotsc, M\}$, then using the Cauchy-Schwarz inequality lead to
		\begin{equation*}
			\begin{aligned}
				\varepsilon\frac{(\Peps^{\ell}, q)}{||q||_{L^{2}}} \lesssim \varepsilon||\Peps^{0}||_{L^{2}} + k\sum_{m=1}^{M}||\nabla \Veps^{m}||_{\mathbb{L}^{2}}, \ \ \forall q \in L^{2}(D)\backslash \{0\}.
			\end{aligned}
		\end{equation*}
		Since $L^{2}(D)$ is the pivot space, the supremum over $q \in L^{2}(D)\backslash \{0\}$ of the left-hand side returns the $L^{2}$-norm of $\varepsilon\Peps^{\ell}$. Therewith, squaring both sides, taking the maximum over $\ell$, employing inequality~\eqref{eq sum-Jensen-type}, applying the mathematical expectation and using Lemma~\ref{lemma a priori estimates}-(i) give $\mathbb{E}\left[\max\limits_{1 \leq m \leq M}\varepsilon^{2}||\Peps^{m}||^{2}_{L^{2}}\right] \lesssim \mathbb{E}\left[\varepsilon^{2}||\Peps^{0}||_{L^{2}}^{2}\right] + kC_{1}$. Dividing by $\varepsilon$ and employing estimate~\eqref{eq projection stability} complete the proof of assertion $(i)$. On the other hand, by arguing in the same way, equation~\eqref{calc17} implies $\varepsilon||\Peps^{m} - \Peps^{m-1}||_{L^{2}} \lesssim k||\nabla \Veps^{m}||_{\mathbb{L}^{2}}$. Squaring both sides, summing over $m$ from $1$ to $M$, taking the mathematical expectation and utilizing Lemma~\ref{lemma a priori estimates}-(i) terminate the proof.
	\end{proof}
	
	The following lemma states a monotonicity property of the operator $u \mapsto -\nu\Delta u + \hat{B}(u,u)$. This feature together with the Lipschitz-continuity of the diffusion coefficient $g$ allow the avoidance of the Skorokhod theorem which forces the filtered probability space that was defined in Section~\ref{section notations and preliminaries} to be exchanged with a new one.
	\begin{lem}\label{lemma monotonicity}
		Assume that $L_{g} \leq \sqrt{\frac{\nu}{2C^{2}_{P}}}$ where $C_{P}>0$ is the Poincar{\'e} constant, and let $u, w \in \mathbb{H}^{1}_{0}$. For $z \coloneqq u -w$, the following inequality holds true:
		\begin{equation*}
			\left\langle -\nu\Delta z + \hat{B}(u,u) - \hat{B}(w, w) + \frac{27}{2\nu^{3}}\left|\left|w\right|\right|^{4}_{\mathbb{L}^{4}}z, z \right\rangle - \left|\left|g(u) - g(w)\right|\right|_{\mathscr{L}_{2}(\sqrt{Q}(K), \mathbb{L}^{2})}^{2} \geq 0.
		\end{equation*}
	\end{lem}
	\begin{proof}
		From \cite[Lemma 2.4]{Menaldi2008Sritharan}, it holds that 
		\begin{equation*}
			\langle -\nu\Delta z + \hat{B}(u, u) - \hat{B}(w, w) + \frac{27}{2\nu^{3}}||w||^{4}_{\mathbb{L}^{4}}z, z\rangle \geq \frac{\nu}{2}||\nabla z||^{2}_{\mathbb{L}^{2}}.
		\end{equation*}
		It suffices now to subtract from both sides the term $||g(u) - g(w)||^{2}_{\mathscr{L}_{2}(\sqrt{Q}(K), \mathbb{L}^{2})}$, use assumption~\ref{S3}, then employ the Poincar{\'e} inequality.
	\end{proof}
	
	Beside Lemma~\ref{lemma monotonicity}, it is worth highlighting the strong convergence of $\{g(\Vkh^{+}) - g(\Vkh^{-})\}_{k,h}$ in $L^{2}(\Omega; L^{2}(0,T; \mathscr{L}_{2}(K, \mathbb{L}^{2})))$, which can be illustrated through assumption~\ref{S3} and Lemma~\ref{lemma a priori estimates}-(i) as follows
	\begin{equation}\label{eq conv g increment}
		\begin{aligned}
			\mathbb{E}\left[\int_{0}^{T}\big|\big|g(\Vkh^{+}) - g(\Vkh^{-})\big|\big|_{\mathscr{L}_{2}(K, \mathbb{L}^{2})}^{2}dt\right] \leq L_{g}^{2}k\mathbb{E}\left[\sum_{m=1}^{M}||\Veps^{m} - \Veps^{m-1}||^{2}_{\mathbb{L}^{2}}\right] \leq L_{g}^{2}C_{1}k \to 0.
		\end{aligned}
	\end{equation}
	The convergence demonstration down below is broken down into steps for clarity's sake.\newline
	\textbf{Step1: Weak convergence and divergence-free}\newline
	By virtue of Lemma~\ref{lemma a priori estimates}, the sublinearity of $g$ (see assumption~\ref{S3}) and inequality~\eqref{eq projection stability}, the sequences $\{\Vkh^{+}\}_{\varepsilon, k, h}$, $\{\sqrt{\varepsilon}\Pkh^{+}\}_{\varepsilon, k, h}$, $\{g(\Vkh^{-})\}_{\varepsilon, k, h}$ are bounded in the Banach spaces $L^{2}(\Omega; L^{\infty}(0,T; \mathbb{L}^{2})\cap L^{2}(0,T; \mathbb{H}^{1}_{0}))$, $L^{2}(\Omega; L^{\infty}(0,T; L^{2}(D)))$ and $L^{2}(\Omega; L^{2}(0,T: \mathscr{L}_{2}(K, \mathbb{L}^{2})))$, respectively. Therefore, the Banach-Alaoglu theorem ensures the existence of the limiting functions $v \in L^{2}(\Omega; L^{\infty}(0, T; \mathbb{L}^{2})\cap L^{2}(0,T; \mathbb{H}^{1}_{0}))$, $\rchi \in L^{2}(\Omega; L^{\infty}(0,T; L^{2}(D)))$, $G_{0} \in L^{2}(\Omega; L^{2}(0,T; \mathscr{L}_{2}(K, \mathbb{L}^{2})))$ and two subsequences (still denoted as their original sequences) $\{\Vkh^{+}\}_{\varepsilon, k, h}$, $\{\sqrt{\varepsilon}\Pkh^{+}\}_{\varepsilon, k, h}$ such that 
	\begin{align}
		&\Vkh^{+} \overset{\ast}{\rightharpoonup} v  &\mbox{ in } \ L^{2}(\Omega; L^{\infty}(0,T; \mathbb{L}^{2})),\label{conv1} \\& \Vkh^{+} \rightharpoonup v &\mbox{ in } \ L^{2}(\Omega; L^{2}(0,T: \mathbb{H}^{1}_{0})),\label{conv2} \\&
		\sqrt{\varepsilon}\Pkh^{+} \overset{\ast}{\rightharpoonup} \rchi &\mbox{ in } \ L^{2}(\Omega; L^{\infty}(0,T; L^{2}(D))),\label{conv3} \\&
		g(\Vkh^{-}) \rightharpoonup G_{0} &\mbox{ in } \ L^{2}(\Omega; L^{2}(0, T; \mathscr{L}_{2}(K, \mathbb{L}^{2}))).\label{conv4}
	\end{align}
	Beside convergence~\eqref{conv4}, it is also possible to acquire $g(\Vkh^{+}) \rightharpoonup G_{0}$ in $L^{2}(\Omega; L^{2}(0,T; \mathscr{L}_{2}(K, \mathbb{L}^{2})))$ as follows: for all $\phi \in L^{2}(\Omega; L^{2}(0,T; \mathscr{L}_{2}(K, \mathbb{L}^{2})))$, 
	\begin{equation}\label{conv5}
		\begin{aligned}
			&\left(g(\Vkh^{+}) - G_{0}(t), \phi(t)\right)_{\mathscr{L}_{2}} = \left(g(\Vkh^{+}) - g(\Vkh^{-}), \phi(t)\right)_{\mathscr{L}_{2}} + \left(g(\Vkh^{-}) - G_{0}(t), \phi(t)\right)_{\mathscr{L}_{2}}.
		\end{aligned}
	\end{equation}
	Now, integrate with respect to $t$, take the mathematical expectation, use results~\eqref{eq conv g increment} and \eqref{conv4} to complete the proof.\newline
	The obtained function $v$ is divergence-free. Indeed, let $q \in C_{c}^{\infty}(D)$ be a scalar function. From Algorithm~\ref{Algo}, one has $\varepsilon\left(\Peps^{m} - \Peps^{m-1}, \rho_{h}q\right) = -k\left(div\Veps^{m}, \rho_{h}q\right)$. Summing both sides over $m$ from $1$ to $M$ leads to $\int_{0}^{T}\left(div\Vkh^{+}, \rho_{h}q\right)dt = \varepsilon\left(\Peps^{0}, \rho_{h}q\right) - \sqrt{\varepsilon}\left(\sqrt{\varepsilon}\Pkh^{+}(T), \rho_{h}q\right).$
	The mathematical expectation of the right-hand side goes to $0$ as $\varepsilon, k, h \to 0$ due to convergence~\eqref{conv3} and estimate~\eqref{eq projection stability}. Hence,
	\begin{equation*}
		\begin{aligned}
			\mathbb{E}\left[\int_{0}^{T}\left(div\Vkh^{+}, q\right)dt\right] = \mathbb{E}\left[\int_{0}^{T}\left(div\Vkh^{+}, q - \rho_{h}q\right)dt\right] + \mathbb{E}\left[\int_{0}^{T}\left(div\Vkh^{+}, \rho_{h}q\right)dt\right]
		\end{aligned}
	\end{equation*}
	converges to $0$ as $\varepsilon, k, h \to 0$, thanks to estimate~\eqref{eq FE error estimate2} and convergence $div\Vkh^{+} \to div(v)$ in $L^{2}(\Omega; L^{2}(0,T; L^{2}(D)))$ which follows straightforwardly from result~\eqref{conv2}. Subsequently, $div\Vkh^{+} \rightharpoonup 0$ in $L^{2}(\Omega; L^{2}(0,T; L^{2}(D)))$ which implies $div(v) = 0$ $\mathbb{P}$-a.s. and a.e. in $(0,T) \times D$. \newline
	Let $\mathscr{R} \colon \mathbb{H}^{1}_{0} \to \mathbb{H}^{-1}$ be defined by $\mathscr{R}(u) \coloneqq -\nu\Delta u + \hat{B}(u, u)$, for all $u \in \mathbb{H}^{1}_{0}$. From Algorithm~\ref{Algo}, and for all $\varphi \in \mathcal{V}$ such that $\varphi_{h} \coloneqq \Pi_{h}\varphi$, it follows
	\begin{equation}\label{calc10}
		\begin{aligned}
			\int_{0}^{T}\langle \mathscr{R}(\Vkh^{+}) + \nabla \Pkh^{+}, \varphi_{h} \rangle dt &= -\left(\Vkh^{+}(T) - \Vkh^{-}(0), \varphi_{h}\right) + \int_{0}^{T}\langle f^{+}, \varphi_{h} \rangle dt \\&\hspace{10pt}+ \left(\int_{0}^{T}g(\Vkh^{-})dW(t), \varphi_{h}\right).
		\end{aligned}
	\end{equation}
	Owing to results~\eqref{conv1} and \eqref{conv4} along with the strong convergence of $f^{+}$ in $L^{2}(\Omega; L^{2}(0,T; \mathbb{H}^{-1}))$ (see \cite[Lemma III.4.9]{temam2001navier}), the mathematical expectation of the right-hand side of equation~\eqref{calc10} is convergent. Therewith, define $\mathscr{R}_{0}$ by
	\begin{equation*}
		\mathbb{E}\left[\int_{0}^{T}\langle \mathscr{R}_{0}(t), \varphi\rangle dt\right] = \lim\limits_{\varepsilon, k, h \to 0}\mathbb{E}\left[\int_{0}^{T}\langle\mathscr{R}(\Vkh^{+}) + \nabla\Pkh^{+}, \Pi_{h}\varphi\rangle dt\right], \ \forall \varphi \in \mathcal{V}.
	\end{equation*}
	Subsequently, the limiting function $v$ satisfies $\mathbb{P}$-a.s. and for all $(t, \varphi) \in [0,T]\times \mathcal{V}$ the following:
	\begin{equation}\label{calc11}
		\left(v(t) - v_{0}, \varphi\right) + \int^{t}_{0}\langle\mathscr{R}_{0}(s), \varphi \rangle ds = \int^{t}_{0}\langle f(s), \varphi\rangle ds + \left(\int_{0}^{t}G_{0}(s)dW(s), \varphi\right).
	\end{equation}
	\textbf{Step2: Identification of $\mathscr{R}_{0}$ and $G_{0}$}\newline
	For $\sigma \in C\Big([0,T], [C_{c}^{\infty}(D)]^{2}\Big)$ and all $m \in \{1, \dotsc, M\}$, denote $\sigma^{+}_{h}(t) \coloneqq \sigma_{h}^{m} = \Pi_{h}\sigma(t_{m})$ and define $\displaystyle r^{+}(t) \coloneqq r^{m} \coloneqq \frac{27}{\nu^{3}}k\sum_{n=1}^{m}\left|\left|\sigma^{n}_{h}\right|\right|^{4}_{\mathbb{L}^{4}}$ for all $t \in (t_{m-1}, t_{m}]$, together with an exponential non-increasing function $\eta \colon [0,T] \to \mathbb{R}$ verifying $\eta(0) = 0$, and having the discrete forms $\eta^{+}(t) \coloneqq \eta^{m} \coloneqq e^{-r^{+}(t)}$ for all $t \in (t_{m-1}, t_{m}]$ and $\eta^{-}(t) \coloneqq \eta^{m-1}$ for all $t \in [t_{m-1}, t_{m})$.
	Setting $\left(\varphi_{h}, q_{h}\right) = \left(\Veps^{m}, \Peps^{m}\right)$ in Algorithm~\ref{Algo}, using Cauchy-Schwarz and Young's inequalities, identity $(a - b, a) = \frac{1}{2}||a||^{2}_{\mathbb{L}^{2}} - \frac{1}{2}||b||^{2}_{\mathbb{L}^{2}} + \frac{1}{2}||a - b||^{2}_{\mathbb{L}^{2}}$, and finally multiplying by $\eta^{m-1}$ yield
	\begin{equation}\label{calc12}
		\begin{aligned}
			&\eta^{m-1}(||\Veps^{m}||^{2}_{\mathbb{L}^{2}} - ||\Veps^{m-1}||^{2}_{\mathbb{L}^{2}}) +2\eta^{m-1}k\left\langle \mathscr{R}(\Veps^{m}) + \nabla\Peps^{m}, \Veps^{m}\right\rangle \leq 2\eta^{m-1}k\langle f^{m}, \Veps^{m}\rangle \\&\hspace{10pt}+ \eta^{m-1}||g(\Veps^{m-1})\Delta_{m}W||^{2}_{\mathbb{L}^{2}} + 2\eta^{m-1}\left(g(\Veps^{m-1})\Delta_{m}W, \Veps^{m-1}\right).
		\end{aligned}
	\end{equation}
	Note that $\sum_{m=1}^{M}\eta^{m-1}(||\Veps^{m}||^{2}_{\mathbb{L}^{2}} - ||\Veps^{m-1}||^{2}_{\mathbb{L}^{2}}) = \int_{0}^{T}\eta^{-}(t)d_{t}||\Vkh^{+}||^{2}_{\mathbb{L}^{2}}dt$, and through equation~\eqref{calc5'}, it holds that $\mathbb{E}\left[||g(\Veps^{m-1})\Delta_{m}W||^{2}_{\mathbb{L}^{2}}\right] = k\mathbb{E}\left[||g(\Veps^{m-1})||^{2}_{\mathscr{L}_{2}(\sqrt{Q}(K), \mathbb{L}^{2})}\right]$. Therefore, taking the sum over $m$ from $1$ to $M$, employing Proposition~\ref{prop discrete derivation}-(ii), then applying the mathematical expectation to equation~\eqref{calc12} give
	\begin{equation}\label{calc13}
		\begin{aligned}
			&\mathbb{E}\left[\eta^{+}(T)||\Vkh^{+}(T)||^{2}_{\mathbb{L}^{2}} - ||\Vkh^{-}(0)||^{2}_{\mathbb{L}^{2}}\right] \leq \mathbb{E}\left[\int_{0}^{T}||\Vkh^{+}||^{2}_{\mathbb{L}^{2}}d_{t}\eta^{+}dt\right] \\&-\mathbb{E}\left[\int_{0}^{T}\eta^{-}(t)\left\langle 2\mathscr{R}(\Vkh^{+}) + 2\nabla\Pkh^{+}, \Vkh^{+} \right\rangle dt\right] + \mathbb{E}\left[2\int_{0}^{T}\eta^{-}(t)\left\langle f^{+}, \Vkh^{+}\right\rangle dt\right] \\&+\mathbb{E}\left[\int_{0}^{T}\eta^{-}(t)||g(\Vkh^{-})||^{2}_{\mathscr{L}_{2}(\sqrt{Q}(K), \mathbb{L}^{2})}dt\right] \eqqcolon I + II + III + IV,
		\end{aligned}
	\end{equation}
	where the last term on the right-hand side of equation~\eqref{calc12} vanishes after taking its expectation due to assumption~\ref{S3} and the measurability of $\{\Veps^{m}\}_{m}$ (see Lemma~\ref{lemma existence+measurability}). By virtue of Proposition~\ref{prop discrete derivation}-(iii), it follows that $d_{t}\eta^{+} = -\frac{27}{\nu^{3}}\eta^{-}||\sigma_{h}^{+}||^{4}_{\mathbb{L}^{4}} + \frac{27^{2}k}{2\nu^{6}}e^{\delta(t)}||\sigma_{h}^{+}||^{8}_{\mathbb{L}^{4}}$, for some $\delta \in (-r^{+}, -r^{-})$. Therefore,
	\begin{equation*}
		\begin{aligned}
			I = -\mathbb{E}\left[\int_{0}^{T}\eta^{-}(t)\frac{27}{\nu^{3}}||\sigma_{h}^{+}||^{4}_{\mathbb{L}^{4}}||\Vkh^{+}||^{2}_{\mathbb{L}^{2}}dt\right] + \frac{27^{2}}{2\nu^{6}}k\mathbb{E}\left[\int_{0}^{T}||\Vkh^{+}||^{2}_{\mathbb{L}^{2}}e^{\delta(t)}||\sigma_{h}^{+}||^{8}_{\mathbb{L}^{4}}dt\right] \eqqcolon I_{1} + I_{2}.
		\end{aligned}
	\end{equation*}
	Obviously, $I_{2}$ goes to $0$ as $k, h, \varepsilon \to 0$ thanks to Lemma~\ref{lemma a priori estimates}. $I_{1}$ can be rewritten as follows
	\begin{equation*}
		\begin{aligned}
			I_{1} &= -\frac{27}{\nu^{3}}\mathbb{E}\left[\int_{0}^{T}\eta^{-}||\sigma_{h}^{+}||^{4}_{\mathbb{L}^{4}}||\Vkh^{+} - \sigma_{h}^{+}||^{2}_{\mathbb{L}^{2}}dt\right] -\frac{27}{\nu^{3}}\mathbb{E}\left[\int_{0}^{T}\eta^{-}||\sigma_{h}^{+}||^{4}_{\mathbb{L}^{4}}\left\{2\left(\Vkh^{+}, \sigma_{h}^{+}\right) - ||\sigma_{h}^{+}||_{\mathbb{L}^{2}}^{2}\right\}dt\right] \\&\eqqcolon I_{1,1} + I_{1,2}.
		\end{aligned}
	\end{equation*}
	Making use of result~\eqref{conv2} along with the strong convergence of $\{\sigma_{h}^{m}\}_{m}$ to $\sigma$ in $C([0,T]; \mathbb{H}^{1}_{0})$, it can be easily shown that $I_{1,2} \to -\frac{27}{\nu^{3}}\mathbb{E}\left[\int_{0}^{T}\eta(t)||\sigma(t)||^{4}_{\mathbb{L}^{4}}\left\{2\Big(v(t), \sigma(t)\Big) - ||\sigma(t)||^{2}_{\mathbb{L}^{2}}\right\}dt\right]$.
	On the other hand,
	\begin{equation*}
		\begin{aligned}
			II &= -\mathbb{E}\left[\int_{0}^{T}\eta^{-}\langle 2\mathscr{R}(\Vkh^{+}) - 2\mathscr{R}(\sigma_{h}^{+}), \Vkh^{+} - \sigma_{h}^{+}\rangle dt\right] - \mathbb{E}\left[\int_{0}^{T}\eta^{-}\langle 2\nabla\Pkh^{+}, \Vkh^{+} - \sigma_{h}^{+} \rangle dt\right] \\&\hspace{10pt}- \mathbb{E}\left[\int_{0}^{T}\eta^{-}\langle 2\mathscr{R}(\Vkh^{+}) + 2\nabla\Pkh^{+} - 2\mathscr{R}(\sigma_{h}^{+}), \sigma_{h}^{+}\rangle dt\right] - \mathbb{E}\left[\int_{0}^{T}\eta^{-}\langle 2\mathscr{R}(\sigma_{h}^{+}), \Vkh^{+}\rangle dt\right] \\&\eqqcolon II_{1} + II_{2} + II_{3} + II_{4}.
		\end{aligned}
	\end{equation*}
	$II_{2}$ goes to $0$ provided $\frac{k}{\varepsilon} \to 0$. Indeed, by Cauchy-Shwarz's inequality, estimate~\eqref{eq sum-Jensen-type} and Lemma~\ref{lemma a priori estimates}, it follows
	\begin{equation}\label{calc13'}
		\begin{aligned}
			&II_{2} = 2\mathbb{E}\left[k\sum_{m=1}^{M}\eta^{m-1}\left(\Peps^{m}, div(\Veps^{m} - \sigma_{h}^{m})\right)\right] \lesssim \mathbb{E}\left[k\sum_{m=1}^{M}||\Peps^{m}||_{L^{2}}||\nabla(\Veps^{m} - \sigma^{m}_{h})||_{\mathbb{L}^{2}}\right] \\&\leq \sqrt{\frac{k}{\varepsilon}}\mathbb{E}\left[\varepsilon\max\limits_{1 \leq m \leq M}||\Peps^{m}||^{2}_{L^{2}}\right]^{\frac{1}{2}}\mathbb{E}\left[3k\sum_{m=1}^{M}||\nabla(\Veps^{m} - \sigma_{h}^{m})||^{2}_{\mathbb{L}^{2}}\right]^{\frac{1}{2}} \lesssim \sqrt{\frac{k}{\varepsilon}}C_{1} \to 0.
		\end{aligned}
	\end{equation}
	Moreover, since $\{\sigma_{h}^{m}\}_{m}$ is strongly convergent toward $\sigma$ in $C([0,T]; \mathbb{H}^{1}_{0})$, and by the definition of operator $\mathscr{R}_{0}$, one obtains $II_{3} \to -\mathbb{E}\left[\int_{0}^{T}\eta(t)\langle 2\mathscr{R}_{0}(t) - 2\mathscr{R}(\sigma(t)), \sigma(t) \rangle dt\right]$ as $k, h, \varepsilon \to 0$. Similarly, $II_{4} \to -\mathbb{E}\left[\int_{0}^{T}\eta(t)\langle 2\mathscr{R}(\sigma(t)), v(t)\rangle dt\right]$, thanks to convergence~\eqref{conv2}. As mentioned in Step 1, $\{f^{m}\}_{m}$ converges strongly toward $f$ in $L^{2}(\Omega; L^{2}(0,T; \mathbb{H}^{-1}))$. The latter together with convergence~\eqref{conv2} imply that $III \to \mathbb{E}\left[2\int_{0}^{T}\eta(t)\left\langle f(t), v(t) \right\rangle dt\right]$. Moving on to term $IV$, it can be reformulated as follows:
	\begin{equation*}
		\begin{aligned}
			&IV = \mathbb{E}\Big[\int_{0}^{T}\eta^{-}\Big\{||g(\Vkh^{-}) - g(\Vkh^{+})||^{2}_{\mathscr{L}_{2}^{Q}} + ||g(\Vkh^{+}) - g(\sigma_{h}^{+})||_{\mathscr{L}_{2}^{Q}}^{2} - ||g(\sigma_{h}^{+})||^{2}_{\mathscr{L}_{2}^{Q}} \\&+ 2\left(g(\Vkh^{+}), g(\sigma_{h}^{+})\right)_{\mathscr{L}_{2}^{Q}} + 2\left(g(\Vkh^{-}) - g(\Vkh^{+}), g(\Vkh^{+})\right)_{\mathscr{L}_{2}^{Q}}\Big\}dt\Big] \coloneqq IV_{1} + \dotsc + IV_{5},
		\end{aligned}
	\end{equation*}
	where $\mathscr{L}_{2}^{Q} \coloneqq \mathscr{L}_{2}(\sqrt{Q}(K), \mathbb{L}^{2})$. From equation~\eqref{eq conv g increment}, it holds that $IV_{1} \to 0$.	Furthermore, Lemma~\ref{lemma monotonicity} yields $I_{1,1} + II_{1} + IV_{2} \leq 0$, the strong convergence of $\{\sigma_{h}^{m}\}_{m}$ together with result~\eqref{conv5} grant both $IV_{3} \to -\mathbb{E}\left[\int_{0}^{T}\eta(t)||g(\sigma(t))||^{2}_{\mathscr{L}_{2}^{Q}}dt\right]$ and $IV_{4} \to \mathbb{E}\left[2\int_{0}^{T}\eta(t)\left(G_{0}(t), g(\sigma(t))\right)_{\mathscr{L}_{2}^{Q}}dt\right]$. Finally, $IV_{5} \to 0$ by virtue of convergences~\eqref{eq conv g increment} and \eqref{conv5}. Putting it all together, equation~\eqref{calc13} becomes
	\begin{equation}\label{calc14}
		\begin{aligned}
			&\lim\limits_{\varepsilon, k, h \to 0}\mathbb{E}\left[\eta^{+}(T)||\Vkh^{+}(T)||^{2}_{\mathbb{L}^{2}} - ||\Vkh^{-}(0)||_{\mathbb{L}^{2}}^{2}\right] \leq \mathbb{E}\left[\int_{0}^{T}\eta'(t)\left\{2\Big(v, \sigma\Big) - ||\sigma||^{2}_{\mathbb{L}^{2}}\right\}dt\right] \\&- 2\mathbb{E}\Big[\int_{0}^{T}\eta(t)\Big\{\langle \mathscr{R}_{0} -\mathscr{R}(\sigma), \sigma\rangle + \langle \mathscr{R}(\sigma) - f(t), v\rangle + \frac{1}{2}||g(\sigma)||^{2}_{\mathscr{L}_{2}^{Q}} - \Big(G_{0}, g(\sigma)\Big)_{\mathscr{L}_{2}^{Q}}\Big\}dt\Big],
		\end{aligned}
	\end{equation}
	where $\eta(t) = \exp\left(-\frac{27}{\nu^{3}}\int_{0}^{t}||\sigma(s)||^{4}_{\mathbb{L}^{4}}ds\right)$. Taking into account that $\mathbb{E}\left[\eta(T)||v(T)||^{2}_{\mathbb{L}^{2}} - ||v_{0}||_{\mathbb{L}^{2}}^{2}\right]$ is smaller than the left-hand side of equation~\eqref{calc14} (thanks to result~\eqref{conv1}) and applying It{\^o}'s formula to the process $(t, v) \mapsto \eta(t)||v||^{2}_{\mathbb{L}^{2}}$ (recall that $v$ satisfies equation~\eqref{calc11}) lead to
	\begin{equation}\label{calc15}
		\begin{aligned}
			&\mathbb{E}\left[\int_{0}^{T}\eta'(t)||v(t) - \sigma(t)||_{\mathbb{L}^{2}}^{2}dt\right] + \mathbb{E}\left[\int_{0}^{T}\eta(t)\big|\big|G_{0}(t) - g(\sigma(t))\big|\big|_{\mathscr{L}_{2}(\sqrt{Q}(K), \mathbb{L}^{2})}^{2}dt\right] \\&\leq 2\mathbb{E}\left[\int_{0}^{T}\eta(t)\big\langle \mathscr{R}(\sigma(t)) - \mathscr{R}_{0}(t), \sigma(t) - v(t)\big\rangle dt\right],\ \ \forall \sigma \in C([0,T]; [C_{c}^{\infty}(D)]^{2}).
		\end{aligned}
	\end{equation}
	Arguing by density, it can be shown that inequality~\eqref{calc15} holds for all $\sigma \in L^{4}(\Omega; L^{\infty}(0,T; \mathbb{L}^{2})) \cap L^{2}(\Omega; L^{2}(0,T; \mathbb{H}^{1}_{0}))$. Therefore, setting $\sigma = v$ yields $G_{0} = g(v)$ $\mathbb{P}$-a.s. and a.e. in $[0,T]\times D$. With that said, the second term on the left-hand side of equation~\eqref{calc15} cancels out. To identify $\mathscr{R}_{0}$, it suffices to consider $\sigma = v + \mu u$ for $\mu > 0$ and $u \in L^{4}(\Omega; L^{\infty}(0,T; \mathbb{L}^{2}))\cap L^{2}(\Omega; L^{2}(0,T; \mathbb{H}^{1}_{0}))$. Subsequently, 
	\begin{equation*}
		\mu\mathbb{E}\left[\int_{0}^{T}\eta'(t)||u(t)||_{\mathbb{L}^{2}}^{2}dt\right] \leq 2\mathbb{E}\left[\int_{0}^{T}\eta(t)\big\langle \mathscr{R}(v(t) + \mu u(t)) -\mathscr{R}_{0}(t), u(t)\big\rangle dt\right].
	\end{equation*}
	Letting $\mu \to 0$ and taking into consideration the hemicontinuity of the operator $\mathscr{R}$, one infers that $\mathbb{E}\left[\int_{0}^{T}\eta(t)\big\langle \mathscr{R}(v(t)) - \mathscr{R}_{0}(t), u(t)\big\rangle dt\right] \geq 0$, for all $u \in L^{4}(\Omega; L^{\infty}(0,T; \mathbb{L}^{2})) \cap L^{2}(\Omega; L^{2}(0,T; \mathbb{H}^{1}_{0}))$. Consequently, $\mathscr{R}_{0} = \mathscr{R}(v)$ in $L^{2}(\Omega; L^{2}(0,T; \mathbb{H}^{-1}))$.\newline
	\textbf{Step 3: Verification of $v$ as NSE solution}\newline
	The obtained function $v$ is henceforth a solution to equations~\eqref{eq Navier-Stokes} in the sense of Definition~\ref{definition NSE sol}. Indeed, the identifications in Step 2 turn equation~\eqref{calc11} into
	\begin{equation*}
		\begin{aligned}
			&\left(v(t) , \varphi\right) + \nu\int_{0}^{t}\left(\nabla v(s), \nabla \varphi\right)ds + \int_{0}^{t}\langle \hat{B}(v(t), v(t)), \varphi\rangle ds \\&= \left(v_{0}, \varphi\right) + \int_{0}^{t}\langle f(s), \varphi \rangle ds + \left(\int_{0}^{t}g(v(s))dW(s), \varphi\right), \ \ \forall \varphi \in \mathbb{V}.
		\end{aligned}
	\end{equation*}	
	By definition, $\hat{B}(v, v) = \left([v\cdot\nabla] + \frac{1}{2}div (v)\right)v = [v\cdot\nabla]v$, thanks to Step 2, where the divergence-free of $v$ was illustrated. Finally, $v \in L^{2}(\Omega; C([0,T]; \mathbb{L}^{2}))$ can be easily proven via equation~\eqref{calc11} by using \cite{Pardoux1975}.\newline
	\textbf{Step 4: Convergence of the whole sequence}\newline
	Convergence results that were discovered within Step 1 are all up to a subsequence. However, due to the uniqueness of $v$ (see \cite[Proposition 3.2]{menaldi2002stochastic}), it follows that the whole sequence $\{\Vkh^{+}\}_{\varepsilon, k,h}$ is convergent toward $v$.
	\subsection{A linear version of Algorithm~\ref{Algo}}\label{section linearized version}
	In terms of simulations, a less time-consuming numerical scheme can be embodied through a linear Algorithm. This can be made up using a linearization of the trilinear term in Algorithm~\ref{Algo} as follows:
	\begin{algo}\label{Algo2}
		Starting from an initial datum $\left(v_{h}^{0}, p_{h}^{0}\right) \in \mathbb{H}_{h}\times L_{h}$, if $\left(\Veps^{m-1}, \Peps^{m-1}\right) \in \mathbb{H}_{h} \times L_{h}$ is known for some $m \in \{1, \dotsc, M\}$, find $\left(\Veps^{m}, \Peps^{m}\right) \in \mathbb{H}_{h} \times L_{h}$ that satisfies $\mathbb{P}$-a.s. the following:
		\begin{equation*}
			\begin{cases}
				\begin{aligned}
					&\left(\Veps^{m} - \Veps^{m-1}, \varphi_{h}\right) + k\nu \left(\nabla \Veps^{m}, \nabla \varphi\right) + k\hat{b}(\Veps^{m-1}, \Veps^{m}, \varphi_{h}) - \left(\Peps^{m}, div\varphi_{h}\right) \\&\hspace{20pt}= k\left\langle f^{m}, \varphi_{h} \right\rangle + \left(g(\Veps^{m-1})\Delta_{m}W, \varphi_{h}\right), \ \ \forall \varphi \in \mathbb{H}_{h},
				\end{aligned}\\
				\frac{\varepsilon}{k}\left(\Peps^{m} - \Peps^{m-1}, q_{h}\right) + \left(div\Veps^{m}, q_{h}\right) = 0, \ \ \forall q_{h} \in L_{h},
			\end{cases}
		\end{equation*}
		where $f^{m}$, $\Delta_{m}W$ are defined in Algorithm~\ref{Algo} and $\left(\Veps^{0}, \Peps^{0}\right) \coloneqq (v_{h}^{0}, p_{h}^{0})$.
	\end{algo}
	Observe that $\hat{b}(\Veps^{m-1}, \Veps^{m}, \Veps^{m}) = 0$, thanks to Proposition~\ref{prop trilinear form}-(ii). Therefore, iterates $\{(\Veps^{m}, \Peps^{m})\}^{M}_{m=1}$ of Algorithm~\ref{Algo2} satisfy Lemmas~\ref{lemma existence+measurability}, \ref{lemma a priori estimates} and they fulfill better uniqueness properties than those of Algorithm~\ref{Algo}, as demonstrated in Lemma\ref{lemma uniqueness linear}. However, due to the infamous properties of $\hat{b}$, the initial datum $v_{h}^{0}$ should undergo a new assumption which consists of a uniform bound in $h$ of $\left|\left|\nabla v_{h}^{0}\right|\right|_{\mathbb{L}^{2}}$, as explained beneath the proof of Lemma~\ref{lemma uniqueness linear}.
	\begin{lem}\label{lemma uniqueness linear}
		Iterates $\{(\Veps^{m}, \Peps^{m})\}^{M}_{m=1}$ of Algorithm~\ref{Algo2} are unique $\mathbb{P}$-a.s. in $\Omega$ and a.e. in $[0,T]\times D$.
	\end{lem}
	\begin{proof}
		Let $\{(\Veps^{m}, \Peps^{m})\}_{m=1}^{M}$ and $\{(U_{\varepsilon}^{m}, P^{m}_{\varepsilon})\}_{m=1}^{M}$ be two solutions to Algorithm~\ref{Algo2} such that $(\Veps^{0}, \Peps^{0}) = (U^{0}_{\varepsilon}, P^{0}_{\varepsilon}) = (v_{h}^{0}, p_{h}^{0})$. Denote $Z^{m}_{\varepsilon} \coloneqq \Veps^{m} - U_{\varepsilon}^{m}$ and $Q_{\varepsilon}^{m} \coloneqq \Peps^{m} - P_{\varepsilon}^{m}$, for all $m \in \{0, 1, \dotsc, M\}$. The following equation is $\mathbb{P}$-a.s. satisfied by $\{(Z_{\varepsilon}^{m}, Q_{\varepsilon}^{m})\}^{M}_{m=1}$:
		\begin{equation}\label{calc16}
			\begin{cases}
				\begin{aligned}
					&\left(Z_{\varepsilon}^{m} - Z_{\varepsilon}^{m-1}, \varphi_{h}\right) + k\nu\left(\nabla Z_{\varepsilon}^{m}, \nabla\varphi_{h}\right) + k\left\langle \hat{B}(\Veps^{m-1}, \Veps^{m}) - \hat{B}(U^{m-1}_{\varepsilon}, U^{m}_{\varepsilon}), \varphi_{h}\right\rangle \\&\hspace{20pt}- k\left(Q_{\varepsilon}^{m}, div\varphi_{h}\right) = \left([g(\Veps^{m-1}) - g(U_{\varepsilon}^{m-1})]\Delta_{m}W, \varphi_{h}\right), \ \ \forall \varphi \in \mathbb{H}_{h},
				\end{aligned}\\
				\frac{\varepsilon}{k}\left(Q_{\varepsilon}^{m} - Q_{\varepsilon}^{m-1}, q_{h}\right) + \left(div Z_{\varepsilon}^{m}, q_{h}\right) = 0, \ \ \forall q_{h} \in L_{h}.
			\end{cases}
		\end{equation}
		For $m=1$, it follows that $g(\Veps^{0}) - g(U^{0}_{\varepsilon}) = 0$ and $\hat{B}(\Veps^{0}, \Veps^{1}) - \hat{B}(U^{0}_{\varepsilon}, U^{1}_{\varepsilon}) = \hat{B}(\Veps^{0} - \Veps^{0}, Z_{\varepsilon}^{1}) = 0$. Hence, setting $(\varphi_{h}, q_{h}) = (Z_{\varepsilon}^{1}, Q_{\varepsilon}^{1})$ in equations~\eqref{calc16} yields $||Z^{1}_{\varepsilon}||^{2}_{\mathbb{L}^{2}} + \varepsilon||Q^{1}_{\varepsilon}||_{L^{2}}^{2} + k\nu||\nabla Z_{\varepsilon}^{1}||^{2}_{\mathbb{L}^{2}} = 0$ which implies $Z_{\varepsilon}^{1} = Q_{\varepsilon}^{1} = 0$ $\mathbb{P}$-a.s. and a.e. in $[0,T]\times D$. Arguing by induction completes the proof.
	\end{proof}
	
	All steps that were conducted in section~\ref{subsection convergence} are applicable to Algorithm~\ref{Algo2}, except for Lemma~\ref{lemma monotonicity} which does not suit the associated bilinear operator $\hat{B}$ since its variables are not identical. Therefore, a slight adjustment should take place, and it consists of the following: \newline
	In Step 1 of section~\ref{subsection convergence}, $\mathscr{R}(\Vkh^{+})$ shall be substituted by a new operator $\mathscr{S}(\Vkh^{-}, \Vkh^{+}) \coloneqq -\nu\Delta \Vkh^{+} + \hat{B}(\Vkh^{-}, \Vkh^{+})$ and $\mathscr{R}_{0}$ by $\mathscr{S}_{0}$ which is defined by
	\begin{equation*}
		\left[\int_{0}^{T}\langle\mathscr{S}_{0}(t), \varphi\rangle dt\right] = \lim\limits_{\varepsilon, k, h \to 0}\mathbb{E}\left[\int_{0}^{T}\langle \mathscr{S}(\Vkh^{-}, \Vkh^{+}) + \nabla\Pkh^{+}, \Pi_{h}\varphi dt \rangle\right], \ \ \forall \varphi \in \mathcal{V}.
	\end{equation*}
	Equation~\eqref{calc13} remains unchanged because $\langle \mathscr{S}(\Vkh^{-}, \Vkh^{+}), \Vkh^{+} \rangle = \langle \mathscr{R}(\Vkh^{+}), \Vkh^{+}\rangle$, thanks to Proposition~\ref{prop trilinear form}-\ref{trilinear ii}. However, when passing to the limit, term $II_{3}$ in Step 2 is not suitable for $\mathscr{S}_{0}$, which is why it can be modified by employing Proposition~\ref{prop trilinear form}-\ref{trilinear ii} as follows:
	\begin{equation*}
		\begin{aligned}
			II'_{3} =& -2\mathbb{E}\left[\int_{0}^{T}\eta^{-}\left\langle \mathscr{S}(\Vkh^{-}, \Vkh^{+}) + \nabla\Pkh^{+} - \mathscr{S}(\sigma_{h}^{-}, \sigma_{h}^{+}), \sigma_{h}^{+}\right\rangle dt\right] \\&-2 \mathbb{E}\left[\int_{0}^{T}\eta^{-}\left\langle \hat{B}\left(\Vkh^{+} - \Vkh^{-}, \Vkh^{+}\right), \sigma_{h}^{+}\right\rangle dt\right] \coloneqq II'_{3,1} + II'_{3,2}.
		\end{aligned}
	\end{equation*}
	$II'_{3,1}$ goes to $-2\mathbb{E}\left[\int_{0}^{T}\eta(t)\langle \mathscr{S}_{0}(t) - \mathscr{S}(\sigma(t), \sigma(t)), \sigma(t) \rangle dt\right]$ as $\varepsilon, k ,h \to 0$. Consequently, the whole proof of section~\ref{subsection convergence} becomes applicable to Algorithm~\ref{Algo2}, provided that $II'_{3,2}$ goes to $0$. To this end, denote $\Zkh = \Vkh^{+} - \Vkh^{-}$ and utilize Proposition~\ref{prop trilinear form}-\ref{trilinear iii} to ensure:
	\begin{equation*}
		\begin{aligned}
			&\int_{0}^{T}\eta(t)\left\langle \hat{B}\left(\Zkh, \Vkh^{+}\right), \sigma_{h}^{+}\right\rangle dt \lesssim \int_{0}^{T}||\Zkh||^{\frac{1}{2}}_{\mathbb{L}^{2}}||\nabla\Zkh||_{\mathbb{L}^{2}}^{\frac{1}{2}}||\nabla\Vkh^{+}||_{\mathbb{L}^{2}}||\nabla \sigma_{h}^{+}||_{\mathbb{L}^{2}}dt \\&\lesssim k^{\frac{1}{4}}\left(\sum_{m=1}^{M}||\Veps^{m} - \Veps^{m-1}||_{\mathbb{L}^{2}}^{2}\right)^{\frac{1}{4}}\left(k\sum_{m=1}^{M}||\nabla (\Veps^{m} - \Veps^{m-1})||^{2}_{\mathbb{L}^{2}}\right)^{\frac{1}{4}}\left(k\sum_{m=1}^{M}||\nabla\Veps^{m}||^{2}_{\mathbb{L}^{2}}\right)^{\frac{1}{2}},
		\end{aligned}
	\end{equation*}
	thanks to the H{\"o}lder inequality and the high regularity of $\sigma$. Therewith, 
	\begin{equation*}
		II'_{3,2} \lesssim k^{\frac{1}{4}}\mathbb{E}\left[\sum_{m=1}^{M}||\Veps^{m} - \Veps^{m-1}||^{2}_{\mathbb{L}^{2}}\right]^{\frac{1}{4}}\mathbb{E}\left[k\sum_{m=1}^{M}||\nabla(\Veps^{m} - \Veps^{m-1})||^{2}_{\mathbb{L}^{2}}\right]^{\frac{1}{4}}\mathbb{E}\left[k\sum_{m=1}^{M}||\nabla\Veps^{m}||^{2}_{\mathbb{L}^{2}}\right]^{\frac{1}{2}}.
	\end{equation*}
	The first and third expectations are bounded by virtue of Lemma~\ref{lemma a priori estimates}-(i). Additionally, the second expectation, after undergoing a triangle inequality, can be controlled in a similar way provided that $||\nabla v_{h}^{0}||_{\mathbb{L}^{2}}$ is uniformly bounded in $h$. Consequently, $II'_{3,2} \lesssim k^{\frac{1}{4}} \to 0$. With being said, an additional theorem can be given.
	\begin{thm}\label{theorem Algo2}
		Let the hypotheses of Theorem~\ref{theorem 1} be fulfilled and $\left|\left|\nabla v_{h}^{0}\right|\right|_{\mathbb{L}^{2}}$ be uniformly bounded in $h$. Then, there exists a discrete stochastic process $\{\left(\Veps^{m}, \Peps^{m}\right)\}_{m=1}^{M}$ that solves Algorithm~\ref{Algo2} and satisfies Lemmas~\ref{lemma existence+measurability}, \ref{lemma a priori estimates}, \ref{lemma uniqueness linear}. Additionally, if $v_{h}^{0} \to v_{0}$ in $L^{2}(\Omega; \mathbb{L}^{2})$ as $h \to 0$ then, Algorithm~\ref{Algo2} converges toward the unique solution of equations~\ref{eq Navier-Stokes} in the sense of Definition~\ref{definition NSE sol}, as $\varepsilon, k, h \to 0$, provided $\frac{k}{\varepsilon} \to 0$.
	\end{thm}
	One way of ensuring uniform boundedness in $h$ of $\left|\left|\nabla v_{h}^{0}\right|\right|_{\mathbb{L}^{2}}$ is through the Ritz (also known as elliptic) operator $\mathcal{R}_{h} \colon \mathbb{H}^{1}_{0} \to \mathbb{H}_{h}$, which is stable in $\mathbb{H}^{1}_{0}$ (see for instance \cite{thomee2007galerkin}). In other words, setting $v_{h}^{0} = \mathcal{R}_{h}v_{0}$ gets the job done, as long as $v_{0} \in \mathbb{H}_{0}^{1}$. Another way is to use the already defined projection $\Pi_{h}$ which can be an alternative for $\mathcal{R}_{h}$. This is true since the triangulation $\mathcal{T}_{h}$ is quasi-uniform (see \cite[Theorem 4]{Crouzeix1987Thomee}).
	\subsection{How to properly choose $\varepsilon$?}\label{section choice of epsilon}
	Beside the condition $\frac{k}{\varepsilon} \to 0$ which was assumed in Section~\ref{section estimates and convergence} to afford the convergence of Algorithm~\ref{Algo}, some choices of $\varepsilon$ do not seem to perform well. For simplicity's sake and knowing that the Stokes problem establishes an insight into the Navier-Stokes equations, the primary aim of this section will be to evaluate a Stokes version of Algorithm~\ref{Algo} against a non-penalty-based numerical scheme of the following stochastic Stokes problem:
	\begin{equation}\label{eq Stokes problem}
		\begin{cases}
			\partial_{t}u - \nu\Delta u + \nabla p = f + g(u)\dot{W}, \ \ &\mbox{ in } (0,T) \times D, \\
			div(u) = 0, \ \ &\mbox{ in } (0,T) \times D, \\
			u(0, \cdot) = v_{0}, \ \ &\mbox{ in } D.
		\end{cases}
	\end{equation}
	in order to choose the parameter $\varepsilon$ effectively. The finite element spaces $\mathbb{H}_{h}$ and $L_{h}$ will be maintained throughout this section, and the discrete LBB (also known as inf-sup) condition 
	\begin{equation}\label{eq LBB}
		\sup\limits_{\varphi_{h} \in \mathbb{H}_{h}}\frac{\left(div\varphi_{h}, q_{h}\right)}{\left|\left|\nabla\varphi_{h}\right|\right|_{\mathbb{L}^{2}}} \geq \beta \left|\left|q_{h}\right|\right|_{L^{2}}, \ \ \forall q_{h} \in L_{h},
	\end{equation}
	will be required since numerical schemes of Stokes and Navier-Stokes problems which deal with saddle point approximations cannot converge without it. The constant $\beta > 0$ does not depend on the mesh size $h$. With that said, it is now meaningful to state the convective-free version of Algorithm~\ref{Algo}:
	\begin{equation}\label{eq convective-free algo}
		\begin{cases}
			\begin{aligned}
				&\left(\Ueps^{m} - \Ueps^{m-1}, \varphi_{h}\right) + k\nu\left(\nabla \Ueps^{m}, \nabla\varphi_{h}\right) - k\left(\pheps^{m}, div\varphi_{h}\right) \\&\hspace{20pt}= k\langle f^{m}, \varphi_{h} \rangle + \left(g(\Ueps^{m-1})\Delta_{m}W, \varphi_{h}\right), \ \ \forall \varphi_{h} \in \mathbb{H}_{h},
			\end{aligned}\\
			\frac{\varepsilon}{k}\left(\pheps^{m} - \pheps^{m-1}, q_{h}\right) + \left(div\Ueps^{m}, q_{h}\right) = 0, \ \ \forall q_{h} \in L_{h},
		\end{cases}
	\end{equation}
	together with the following saddle point-based numerical scheme of the Stokes problem:
	\begin{equation}\label{eq standard Stokes algo}
		\begin{cases}
			\begin{aligned}
				&\left(U^{m} - U^{m-1}, \varphi_{h}\right) + k\nu\left(\nabla U^{m}, \nabla\varphi_{h}\right) -k\left(p^{m}, div\varphi_{h}\right) \\&\hspace{20pt}= k\left\langle f^{m}, \varphi_{h}\right\rangle + \left(g(U^{m-1})\Delta_{m}W, \varphi_{h}\right), \ \ \forall \varphi_{h} \in \mathbb{H}_{h},
			\end{aligned}\\
			\left(div U^{m}, q_{h}\right) = 0, \ \ \forall q_{h} \in L_{h}.
		\end{cases}
	\end{equation}
	Here, $\Delta_{m}W$ and $f^{m}$ are identical to those of Algorithm~\ref{Algo}, and the starting points $\Ueps^{0} = U^{0} = \Pi_{h}v_{0}$. The convergence analysis of scheme~\eqref{eq standard Stokes algo} along with its convergence rate are provided in \cite{Feng2020Qiu}. To come up with effective and adequate conditions upon the parameter $\varepsilon$, it suffices to investigate the quantity $\left|\left|\Ueps^{m} - U^{m}\right|\right|$. This is logical because if $u$ denotes the solution of Stokes equations~\eqref{eq Stokes problem}, then $\left|\left|\Ueps^{m} - u(t_{m})\right|\right| \leq \left|\left|\Ueps^{m} - U^{m}\right|\right| + \left|\left|U^{m} - u(t_{m})\right|\right|$ grants the rate at which scheme~\eqref{eq convective-free algo} might converge. To this purpose, subtracting equations~\eqref{eq convective-free algo} and \eqref{eq standard Stokes algo} yields
	\begin{equation}\label{calc18}
		\begin{aligned}
			&\left(\Ueps^{m} - U^{m} - (\Ueps^{m-1} - U^{m-1}) - [g(\Ueps^{m-1}) - g(U^{m-1})]\Delta_{m}W, \varphi_{h}\right) + k\left(p^{m} - \pheps^{m}, div\varphi_{h}\right) \\&= k\nu\left(\nabla(U^{m} - \Ueps^{m}), \nabla \varphi_{h}\right) \leq k\nu||\nabla (U^{m} - \Ueps^{m})||_{\mathbb{L}^{2}}||\nabla \varphi_{h}||_{\mathbb{L}^{2}}, \ \ \forall \varphi_{h} \in \mathbb{H}_{h}\backslash \{0\}.
		\end{aligned}
	\end{equation}
	Dividing by $||\nabla \varphi_{h}||_{\mathbb{L}^{2}}$, taking the supremum over $\varphi_{h} \in \mathbb{H}_{h}\backslash \{0\}$ and employing the discrete LBB-condition~\eqref{eq LBB} imply
	\begin{equation}\label{calc19}
		||p^{m} - \pheps^{m}||_{L^{2}} \leq \frac{\nu}{\beta}||\nabla(U^{m} - \Ueps^{m})||_{\mathbb{L}^{2}}, \ \ \forall m \in \{1, \dotsc, M\}.
	\end{equation}
	Estimate~\eqref{calc19} is true because $\omega \mapsto \sup\limits_{\varphi_{h} \in \mathbb{H}_{h}}\frac{\left(\Ueps^{m} - U^{m} - (\Ueps^{m-1} - U^{m-1}) - [g(\Ueps^{m-1}) - g(U^{m-1})]\Delta_{m}W, \varphi_{h}\right)}{||\nabla \varphi_{h}||_{\mathbb{L}^{2}}}$ is non-negative which results from the fact that $\mathbb{H}_{h}$ is a vector space. In other words, this supremum can be roughly seen as the $\mathbb{H}^{-1}$-norm of $\Ueps^{m} - U^{m} - (\Ueps^{m-1} - U^{m-1}) - [g(\Ueps^{m-1}) - g(U^{m-1})]\Delta_{m}W$. On the other hand, setting $\varphi_{h} = \Ueps^{m} - U^{m}$ in equation~\eqref{calc18}, using identity $2(a - b,a) = ||a||^{2}_{\mathbb{L}^{2}} - ||b||^{2}_{\mathbb{L}^{2}} + ||a-b||^{2}_{\mathbb{L}^{2}}$, the Cauchy-Schwarz and Young inequalities return
	\begin{equation}\label{calc20}
		\begin{aligned}
			&\frac{1}{2}||\Ueps^{m} - U^{m}||^{2}_{\mathbb{L}^{2}} - \frac{1}{2}||\Ueps^{m-1} - U^{m-1}||^{2}_{\mathbb{L}^{2}} + k\nu||\nabla(\Ueps^{m} - U^{m})||_{\mathbb{L}^{2}}^{2} \leq k\left(\pheps^{m} - p^{m}, div\Ueps^{m}\right) \\&+ \left([g(\Ueps^{m-1}) - g(U^{m-1})]\Delta_{m}W, \Ueps^{m-1} - U^{m-1}\right) + \frac{1}{2}\left|\left|[g(\Ueps^{m-1}) - g(U^{m-1})]\Delta_{m}W\right|\right|^{2}_{\mathbb{L}^{2}},
		\end{aligned}
	\end{equation}
	where $(\pheps^{m} - p^{m}, div U^{m}) = 0$, thanks to scheme~\eqref{eq standard Stokes algo}. Summing the above equation over $m$ from $1$ to an arbitrary $\ell \in \{1, \dotsc, M\}$, taking its mathematical expectation, employing the It{\^o} isometry to the last term on its right-hand side together with assumption~\ref{S3} and making use of the identity $\Ueps^{0} = U^{0}$ yield
	\begin{equation}\label{calc21}
		\begin{aligned}
			&\mathbb{E}\left[\frac{1}{2}||\Ueps^{\ell} - U^{\ell}||_{\mathbb{L}^{2}}^{2} + k\nu\sum_{m=1}^{\ell}||\nabla (\Ueps^{m} - U^{m})||^{2}_{\mathbb{L}^{2}}\right] \leq \mathbb{E}\left[k\sum_{m=1}^{\ell}\left(\pheps^{m} - p^{m}, div\Ueps^{m}\right)\right] \\&+ \frac{L_{g}^{2}}{2}\mathbb{E}\left[k\sum_{m=1}^{\ell}||\Ueps^{m-1} - U^{m-1}||^{2}_{\mathbb{L}^{2}}\right],
		\end{aligned}
	\end{equation}
	where the mathematical expectation of the penultimate term in equation~\eqref{calc20} vanishes due to assumption~\ref{S3} and the measurability of $\{\Ueps^{m}\}_{m=1}^{M}$ and $\{U^{m}\}^{M}_{m=1}$. Attention will now turn toward the first term on the right-hand side of equation~\eqref{calc21} which will eventually hand the upper-bound in terms of $\varepsilon$. Using equations~\eqref{eq convective-free algo}, one obtains
	\begin{equation}\label{calc22}
		\begin{aligned}
			&J \coloneqq \mathbb{E}\left[k\sum_{m=1}^{\ell}\left(\pheps^{m} - p^{m}, div\Ueps^{m}\right)\right] = -\mathbb{E}\left[\varepsilon\sum_{m=1}^{\ell}\left(\pheps^{m} - \pheps^{m-1}, \pheps^{m} - p^{m}\right)\right] \\&\leq \sqrt{\varepsilon}\mathbb{E}\left[3\varepsilon\sum_{m=1}^{\ell}||\pheps^{m} - \pheps^{m-1}||^{2}_{L^{2}}\right]^{\frac{1}{2}}\mathbb{E}\left[\max\limits_{1 \leq m \leq \ell}||\pheps^{m}- p^{m}||^{2}_{L^{2}}\right]^{\frac{1}{2}} \\&\leq \frac{\sqrt{3C_{1}}\nu}{\beta}\sqrt{\varepsilon}\mathbb{E}\left[\max\limits_{1 \leq m \leq M}||\nabla(\Ueps^{m} - U^{m})||^{2}_{\mathbb{L}^{2}}\right]^{\frac{1}{2}} \leq \frac{\sqrt{3\mathscr{C}}C_{1}\nu}{\beta}\frac{\sqrt{\varepsilon}}{h},
		\end{aligned}
	\end{equation}
	thanks to the Cauchy-Schwarz inequality, estimate~\eqref{eq sum-Jensen-type}, \eqref{calc19}, Lemma~\ref{lemma a priori estimates}-(i)-(ii), and the inverse inequality~\eqref{eq inverse inequality}. The bound of $\mathbb{E}\left[\max\limits_{1 \leq m \leq M}||U^{m}||_{\mathbb{L}^{2}}^{2}\right]$ in equation~\eqref{calc22} is not carried out herein, but can be found for instance in \cite[Lemma 3.1]{brzezniak2013finite}. Finally, plug the result of equation~\eqref{calc22} in estimate~\eqref{calc21} and make use of the discrete Gr{\"o}nwall inequality to achieve
	\begin{equation}\label{calc23}
		\frac{1}{2}\max\limits_{1 \leq m \leq M}\mathbb{E}\left[\left|\left|\Ueps^{m} - U^{m}\right|\right|_{\mathbb{L}^{2}}^{2}\right] + \mathbb{E}\left[k\nu\sum_{m=1}^{M}\left|\left|\nabla(\Ueps^{m} - U^{m})\right|\right|^{2}_{\mathbb{L}^{2}}\right] \leq \tilde{C}\frac{\sqrt{\varepsilon}}{h},
	\end{equation}
	for some constant $\tilde{C} > 0$ depending only on $\beta, C_{1}, \nu, L_{g}, \mathscr{C}$ and $T$.
	
	Estimate~\eqref{calc23} appears to have the best upper-bound amongst the other possible ways of estimation. Besides, some calculation techniques may be inconsistent with the assumption $\frac{k}{\varepsilon} \to 0$. For instance, by Young's inequality, it holds that $J \leq \sqrt{\varepsilon}\mathbb{E}\left[\sum_{m=1}^{\ell}\sqrt{\varepsilon}\left|\left|\pheps^{m} - \pheps^{m-1}\right|\right|^{2}_{L^{2}}\right] + \frac{\nu^{2}\sqrt{\varepsilon}}{\beta^{2}}\mathbb{E}\left[\sum_{m=1}^{\ell}\left|\left|\nabla (\Ueps^{m} - U^{m})\right|\right|^{2}_{\mathbb{L}^{2}}\right]$, thanks to estimate~\eqref{calc19}. The first term is bounded by $C_{1}\sqrt{\varepsilon}$ by virtue of Lemma~\ref{lemma a priori estimates}. However, the second term needs to be absorbed in the left-hand side of equation~\eqref{calc21}. To this end, the assumption $k\nu - \frac{\nu^{2}\sqrt{\varepsilon}}{\beta^{2}} > 0$ must be imposed. In other words, $\sqrt{\frac{k}{\varepsilon}} > \frac{\nu}{\beta^{2}\sqrt{k}}$, which obviously fulfills the opposite of $\frac{k}{\varepsilon} \to 0$.
	\section{Numerical experiments and conclusion}\label{section numerical exp.}
	The implementation within this section will be carried out through Algorithm~\ref{Algo2} and a saddle point based-numerical scheme \cite[Algorithm 3]{brzezniak2013finite}: 
	\begin{algo}\label{Algo Carelli}
		Let $M \in \mathbb{N}$ and $V^{0} = v_{h}^{0} \in \mathbb{H}_{h}$ be given. For every $m \in \{1, \dotsc, M\}$, find an $\mathbb{H}_{h} \times L_{h}$-valued $\left(V^{m}, \Pi^{m}\right)$ such that
		\begin{equation*}
			\begin{cases}
				\begin{aligned}
					&\left(V^{m} - V^{m-1}, \varphi_{h}\right) + k\nu\left(\nabla V^{m}, \nabla\varphi_{h}\right) + k\hat{b}(V^{m-1}, V^{m}, \varphi_{h}) - k\left(\Pi^{m}, div\varphi_{h}\right) \\&\hspace{20pt}= k\langle f^{m}, \varphi_{h} \rangle + \left(g(V^{m-1})\Delta_{m}W, \varphi_{h}\right), \ \ \forall \varphi_{h} \in \mathbb{H}_{h},
				\end{aligned}\\
				\left(divV^{m}, q_{h}\right) = 0, \ \ \forall q_{h} \in L_{h},
			\end{cases}
		\end{equation*}
	\end{algo}
	which will play the reference role with respect to the values of the parameter $\varepsilon$. The domain's meshing is carried out through the open source finite element mesh generator Gmsh~\cite{geuzaine2009gmsh}, the implementation of the aforementioned algorithms is executed by the open source finite element software FEniCS~\cite{LoggMardalEtAl2012a}, and the visualization is ensured via Paraview~\cite{Ahrens2005Geveci}. The simulation's configuration down below is set as follows: $T = 1$, $\nu = 1$, $h = 0.16$, $\varepsilon = h^{2 + \delta}$, $k = \varepsilon^{1 + \delta}$, with $\delta = 0.1$ to guarantee that $\frac{k}{\varepsilon} \to 0$ and $\frac{\sqrt{\varepsilon}}{h}$ is small enough. For the sake of comparison, the space discretization will be conducted by the lower order Taylor-Hood ($P_{2}/P_{1}$) finite element for both algorithms~\ref{Algo2} and \ref{Algo Carelli}. The initial data $u_{0}$ and $p_{0}$ are set to $0$ which means that $v_{h}^{0}= p_{h}^{0} = 0$. The domain $D$ is an $L$-shaped geometry whose figure and mesh are the following:
	\begin{figure}[H]
		\centering
		\begin{subfigure}{0.375\linewidth}
			\vspace{5pt}\includegraphics[width=0.95\linewidth]{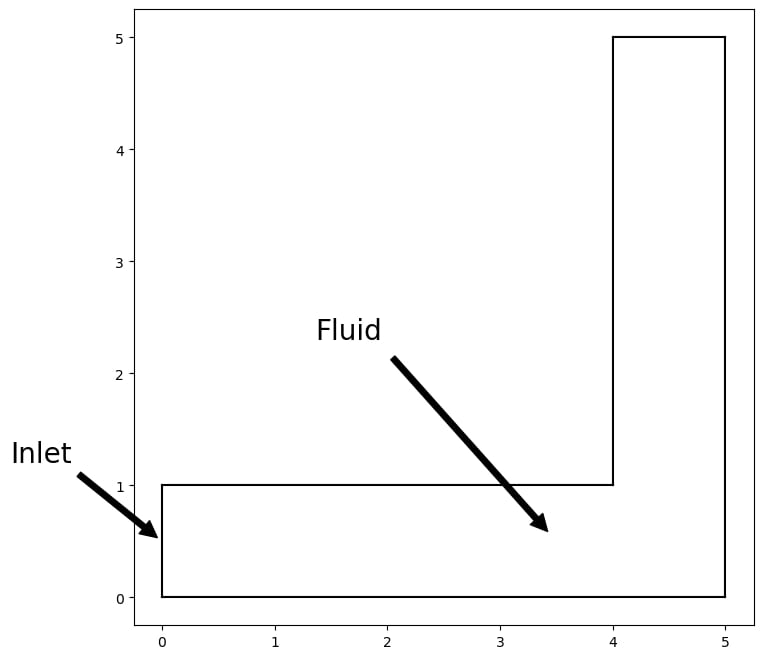}
		\end{subfigure}
		\begin{subfigure}{0.4\linewidth}
			\includegraphics[width=0.94\linewidth]{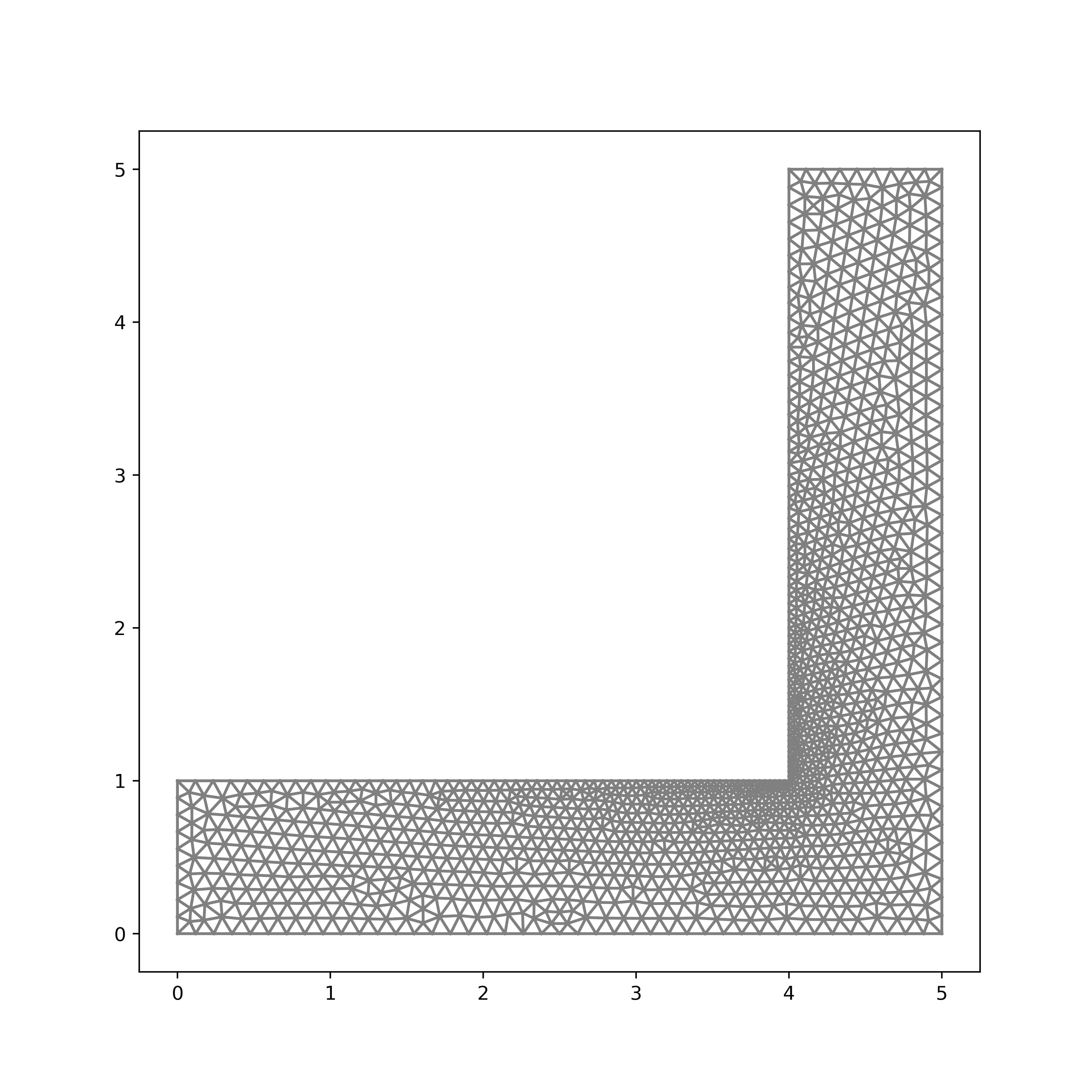}
		\end{subfigure}
		\caption{The domain $D$ and its mesh}
	\end{figure}
	\hspace{-18pt}The boundary condition 
	\begin{equation*}
		v(x, y) = 
		\begin{cases}
			(1, 0) &\mbox{ if } (x, y) \in \{0\}\times [0, 1],\\
			(0, 0) &\mbox{ elsewhere},
		\end{cases}
	\end{equation*}
	is non-homogeneous, which is possible since a simple lifting technique can take the problem's boundary condition back to a homogeneous setting. The source term $f$ takes on the value $(0, 0)$ and the diffusion coefficient $g = 1$ i.e. it is an additive noise. The Wiener increment $\Delta_{m}W$ is approximated as follows: let $J \in \mathbb{N}$ be non-zero, and $W_{1}$, $W_{2}$ be two independent $H^{1}_{0}(D)$-valued Wiener processes  such that $W = (W_{1}, W_{2})$. Then,
	\begin{equation*}
		\Delta_{m}W_{\ell} \approx \sqrt{k}\sum_{i, j =1}^{J}\sqrt{\lambda_{i,j}^{\ell}}\xi_{i,j}^{\ell, m}e_{i,j}, \ \ \ell \in \{1, 2\}.
	\end{equation*}
	The parameter $J$ takes on the value $5$, $\lambda^{\ell}_{i,j} = \frac{1}{(i + j)^{2}}$ for all $i, j \in \mathbb{N}$, $\left\{(\xi_{i,j}^{1, m}, \xi_{i,j}^{2, m})\right\}_{i, j}^{m}$ is a family of independent identically distributed normal random variables, and $e_{i,j}(x,y) = \frac{2}{5}sin(i\pi x/5)sin(j\pi y/5)$ for all $i, j \in \mathbb{N}$. Although $\{e_{i,j}\}_{i,j}$ may not be the best choice for an $L$-shaped domain (because they represent the Laplace eigenfunctions on the square $(0, 5)^{2}$ with a Dirichlet boundary condition), they can be thought of herein as a restriction to $D$. The explicit formula of the Laplace eigenfunctions on an $L$-shaped domain is unknown as it is explained in \cite{Reid1965}. With all that being said, it is now possible to exhibit the simulation results:
	\begin{figure}[H]
		\centering
		\begin{subfigure}{0.3\linewidth}
			\includegraphics[width=\linewidth]{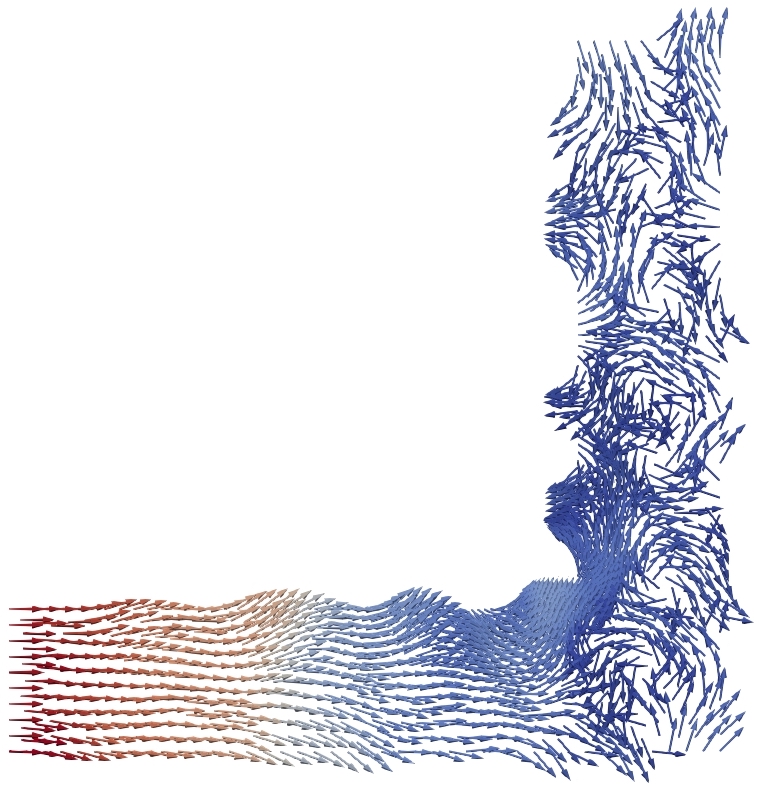}
		\end{subfigure}
		\hspace{20pt}\begin{subfigure}{0.3\linewidth}
			\includegraphics[width=\linewidth]{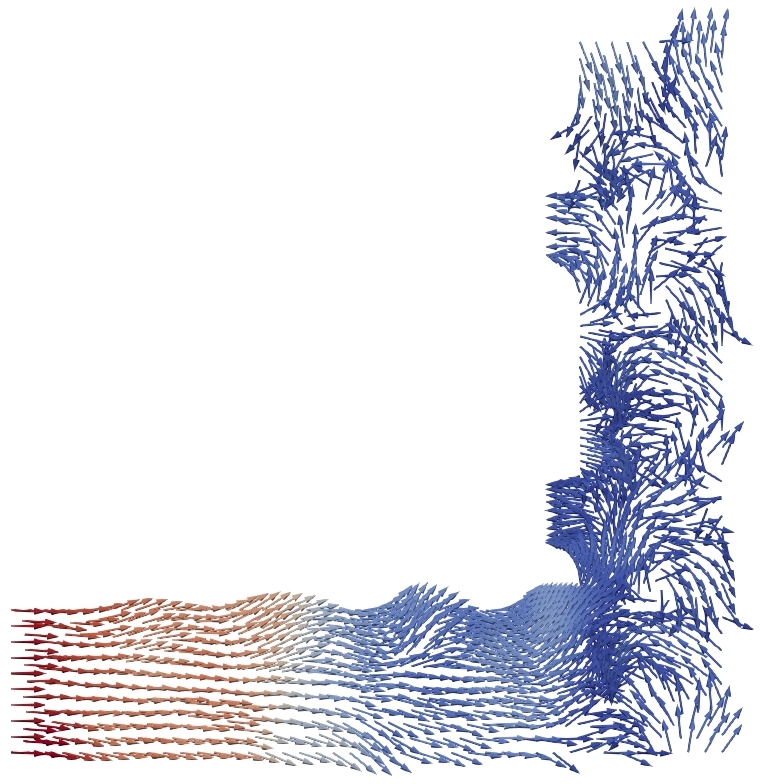}
		\end{subfigure}
		\begin{subfigure}{0.2\linewidth}
			\centering
			\includegraphics[width=0.41\linewidth]{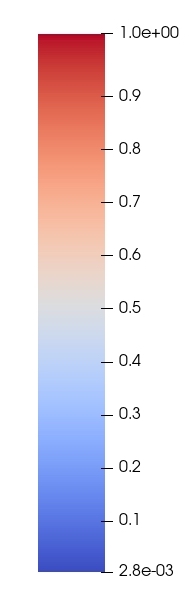}
		\end{subfigure}
		\caption{One realization of $V^{M}$ (left) and $\Veps^{M}$ (right) at time $T=1$ for $\varepsilon = 0.021$}
	\end{figure}
	\hspace{-18pt}As $\varepsilon$ gets smaller, the difference between $V^{m}$ and $\Veps^{m}$ becomes indistinguishable. This fact is illustrated in an accurate way down below where the relationship between $\varepsilon$ and the error $\mathbb{E}\left[\left|\left|V^{M} - \Veps^{M}\right|\right|^{2}_{\mathbb{L}^{2}}\right]$ is exposed:
	\begin{figure}[H]
		\centering
		\begin{minipage}{0.45\linewidth}
			\includegraphics[width=\linewidth]{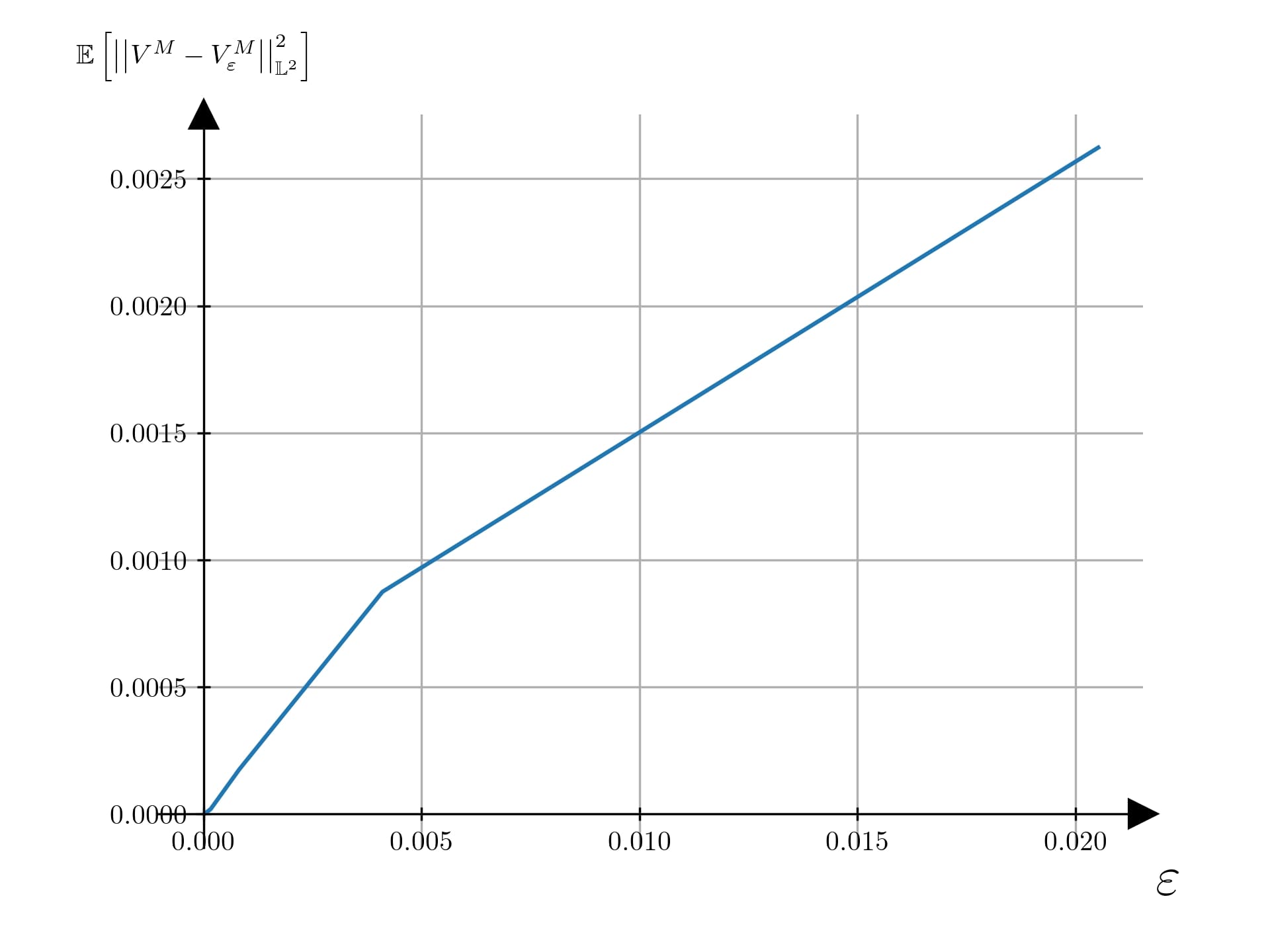}
		\end{minipage}
		\begin{minipage}{0.4\linewidth}
			\centering
			\small
			\hspace{20pt}\begin{tabular}{ c|c }
				& $Var\left(\left|\left|V^{M} -\Veps^{M}\right|\right|_{\mathbb{L}^{2}}\right)$ \\
				\hline
				$\varepsilon$ & $2.9\times 10^{-4}$ \\
				$\varepsilon/5$ & $1.07\times 10^{-4}$ \\
				$\varepsilon/25$ & $2.5\times 10^{-5}$ \\
				$\varepsilon/125$ & $3.2\times 10^{-6}$ \\
				$\varepsilon/625$ & $2.7\times 10^{-7}$
			\end{tabular}
		\end{minipage}
		\caption{Error and error-variance in terms of $\varepsilon$}\label{figure Graph}
	\end{figure}
	\hspace{-18pt}The computed error in figure~\ref{figure Graph} uses a Monte-Carlo method with $1000$ realizations. The obtained curve was expected; it emphasizes the fact that $\varepsilon$ should be taken as small as possible in order to guarantee accurate outcomes.
	\paragraph{\textbf{Acknowledgment}:} The author would like to thank Prof. Ludovic Gouden{\`e}ge (Ecole CentraleSup{\'e}lec) for his valuable comments.
	
	\printbibliography
\end{document}